\documentclass[11pt]{article}
\usepackage{amsmath, amsthm, amscd, amsfonts, amssymb, graphicx, color}
\usepackage[margin=2.8cm]{geometry}
\usepackage[bookmarksnumbered, colorlinks, plainpages]{hyperref}

\newtheorem{theorem}{Theorem}[section]
\newtheorem{lemma}[theorem]{Lemma}
\newtheorem{proposition}[theorem]{Proposition}
\newtheorem{corollary}[theorem]{Corollary}
\theoremstyle{definition}
\newtheorem{definition}[theorem]{Definition}
\newtheorem{example}[theorem]{Example}

\theoremstyle{remark}
\newtheorem{note}{Note}

\numberwithin{equation}{section}

\allowdisplaybreaks
\begin{document}
\title{\textbf{Characterization of the Skew cyclic codes over $\mathbb{F}_p+v\mathbb{F}_p$}}
\author{ Reza DastBasteh{{*}}, Seyyed Hamed Mousavi{{*}{*}}, Javad Haghighat{{*}{*}} \\
{*}Department of Mathematics, Sabanci university, Istanbul, Turkey\\
{{*}{*}}Department of Electrical Engineering, Shiraz University of Technology, Shiraz, Iran\\
e-mail: haghighat@sutech.ac.ir, r.dastbasteh@sabanciuniv.edu, h.moosavi@sutech.ac.ir}
\maketitle
\begin{abstract}  
We study skew cyclic codes with arbitrary length over $\mathbb{F}_p+v\mathbb{F}_p$ where $\theta(v)=\alpha v, \alpha\in \mathbb{F}_p$. We characterize all existing codes in case of $O(\theta)|n$ by using certain projections from $(\mathbb{F}_p+v\mathbb{F}_p)[x;\theta]$ to $\mathbb{F}_p[x]$. We provide an explicit expression for the ensemble of all possible codes. We also prove useful properties of these codes in the case of $O(\theta)\nmid n$. We provide results and examples to illustrate how the codes are constructed and how their encoding and decoding are realized.   
\end{abstract}

\begin{scriptsize}
\begin{tiny}
\end{tiny}
\end{scriptsize}
\section{Introduction}

We study construction and charcteristics of skew cyclic codes over
\textbf{$\mathbb{F}_{p}+v\mathbb{F}_{p}$}. 
The motivation behind studying this specific ring is as follows. Since
$v^{2}=0$, the constructed codes provide the possibility to seperate
the data into two distinct blocks. The other advantage is that the
first type ideals of this ring are partitioned into two explicit parts. The decoding
and encoding of these codes is simpler than the skew cyclic codes
over $\mathbb{F}_{p^{2}}$. This ring is not an ED, so it has a disadvantage
that we cannot use the straightforward algorithm applied for the cyclic
codes. But a similar property holds in this ring suggests some efficient algorithm to encode or decode over this ring. We will prove them in the following sections. It also has a potential advantage that the constructed codes
may be more resilient against channel burst errors. The results of
this paper can be generalized to the ring $(\mathbb{F}_{p}+v\mathbb{F}_{p}+\cdots+v^{n-1}\mathbb{F}_{p})[x;\theta]$
where $v^{n}=0$. 

The ring is not a UFD. The best advantage of this code is that the polynomial $x^{n}-1$ has more divisors
in the ring $(\mathbb{F}_{p}+v\mathbb{F}_{p})[x;\theta]$ than in the
ring $(\mathbb{F}_{p}+v\mathbb{F}_{p})[x]$. Therefore, it should be
possible to find codes with large minimum Hamming distances. (See for exampple \cite{Boucher2008} ,\cite{taher2010})

Our contributions are as follows.
We study the structure of $\frac{(\mathbb{F}_{p}+v\mathbb{F}_{p})[x;\theta]}{<x^{n}-1>}$
where $v^{2}=0$ and $\theta(v)=\alpha v$. This ring is hard to study,
because the homorphism is not Frobinious which is more studied in the literature. We studied this module in both cases $O(\theta)|n$ or $O(\theta)\nmid n$.
The studies of primary decomposition for non-commutative ring is very
rare in the pure math literature. The results of this study are
in the explicit form of decomposition for all of codes with the same
type. We also study the ring $(\mathbb{F}_{p}+v\mathbb{F}_{p})[x;\theta]$
which is interesting from the algebraic point of view.
Moreover, We defined a new kind of projections which can be used in the similar
rings. These projections are very useful in the study of cyclic codes.

Similar works on cyclic codes over rings appear in the literature
including \cite{calderbank95}-\cite{mostafanasab}. 

We close this section by introducing some terms and definitions that
are applied in the following sections. Some of these definitions are
borrowed from the literature, as mentioned above.

\begin{definition} 
Let $R$ be a ring and $P$ be an ideal of it.
Then $P$ is prime, if and only if $AB\subseteq P$ implies $A\subseteq P$
or $B\subseteq P$. 
\end{definition}

\begin{definition} 
Let $R$ be a ring and $P$ be an ideal of it.
Then $Q$ is primary, if and only if there exists $m\in\mathbb{N}$
such that $AB\subseteq Q$ implies $A\subseteq Q$ or $B^{m}\subseteq Q$.
\end{definition}

Let $S$ be a commutative ring with identity and $\theta$ be an automorphism
of $S$. The set $S[x;\theta]=\{\sum_{i=0}^{n}a_{i}x^{i}|a_{i}\in S,n\in\mathbb{N}^{*}\}$
of polynomials forms a ring under usual addition of polynomials and
where multiplication is defined by $xa=\theta(a)x$.

The ring $S[x;\theta]$ is called the skew polynomial ring over $S$.
One can see that $S[x;\theta]$ is non-commutative unless $\theta$
is the identity automorphism of $S$.

\begin{definition} 
Skew cyclic codes over an arbitrary ring $S$
are the linear codes $\complement$ which satisfies the following
properties. Suppose that $\theta$ is an endomorphism of $S$. If
$(c_{0},c_{1},c_{2},\cdots,c_{n})\in\complement$ implies $(\theta(c_{n}),\theta(c_{0}),\cdots,\theta(c_{n-1}))$.
These codes are in fact the submodules of $\frac{S[x;\theta]}{<x^{n}-1>}$.
\end{definition}

From now on $p\in\mathbb{Z}$ is an odd prime, $\mathbb{F}_{p}$ will
denote the Galois field of order $p$, $\mathbb{F}_{p}^{*}=\mathbb{F}_{p}-\{0\}$
and $S=\mathbb{F}_{p}+v\mathbb{F}_{p}$, where $v^{2}=0$. That is,
$S=\{a+bv|a,b\in\mathbb{F}_{p},v^{2}=0\}$. Also $R=S[x;\theta]$,
where $\theta\in Aut(S)$.

\begin{definition} 
We define $R$ as $\mathbb{F}_{p}+v\mathbb{F}_{p}$
where $v^{2}=0$. This is infact the ring $\frac{\mathbb{F}_{p}[v]}{<v^{2}>}$.
\end{definition}

We shall determine $Aut(S)$. Since $\theta(a)=a$ for $\theta\in Aut(S)$
and $a\in\mathbb{F}_{p}$, we have to find the image of $v$ under
such $\theta$. Let $\theta(v)=r+vs$, for $r,s\in\mathbb{F}_{p}$.
Using the additive and the multiplicative properties of $\theta$
one can get that $s=0$ and $r\in\mathbb{F}_{p}^{*}$. Hence $\theta(a+bv)=a+b\alpha v$
for $\alpha\in\mathbb{F}_{p}^{*}$. That is, $Aut(S)=\mathbb{F}_{p}^{*}$.

\begin{definition} 
Let $T$ be a ring. We denote the subring $Z(T)$
as the center of $T$ as follows. 
\begin{align*}
Z(T)=\{r\in T|\forall t\in T:tr=rt\}.
\end{align*}
\end{definition}

\section{Over the ring $(\mathbb{F}_p+v\mathbb{F}_p)[x;\theta]$}
\subsection{The center and units of $R$}

From now on, $\theta$ will denote an automorphism of $S$ of order
$o(\theta)=|\langle\theta\rangle|=e>1$.

Since $R$ is a non-commutative ring, it is worth to find its center.

\begin{theorem} 
The center of $R=S[x;\theta]$ is $\mathbb{F}_{p}[x^{e}]$
for any $\theta\in Aut(S)$ of order $e$. 
\end{theorem}

\begin{proof} 
Let $Z(R)$ be the center of $R$. Since $\theta$
is a non-identity automorphism, $\theta(a+bv)=a+b\alpha v$ for $a,b\in\mathbb{F}_{p}$
and $\alpha\neq0,1$. For $a,b\in\mathbb{F}_{p}$, $(a+bv)x\neq x(a+bv)=(a+b\alpha v)x$.
So $a+bv$ does not commute with $x$. Similarly, $x^{k}$, where
$e\nmid k$ is not in $Z(R)$. We show below that, the element $(a+bv)x^{n}$,
where $e|n$ and $b\neq0$ is not in $Z(R)$. We show that $(1+v)x(a+bv)x^{n}\neq(a+bv)x^{n}(1+v)x$.
Note that $(1+v)x(a+bv)x^{n}=(a+(\alpha^{-1}b+a)v)x^{n+1}$. while,
$(a+bv)x^{n}(1+v)x=(a+(b+a)v)x^{n+1}$. Since $\alpha\neq0,1$, $\alpha^{-1}bv\neq bv$.
So we have shown that all polynomials of the form $f=g+(a+bv)x^{n}$,
where $g\in R$, $e|n$ and $b\neq0$ are not in $Z(R)$.

Note that $\mathbb{F}_{p}$ is fixed by $\theta$. Now, since $o(\theta)=e$,
$x^{ei}a=(\theta^{e})^{i}ax^{ei}=ax^{ei}$ for any $i$ and $a\in\mathbb{F}_{p}$.
So $x^{ie}\in Z(R)$. Moreover, if $f=\sum_{i=0}^{n}a_{i}x^{ei}$
and $g_{1}+vg_{2}\in S$, then 
\begin{align}
f(g_{1}+vg_{2})=fg_{1}+v(a_{0}g+a_{1}\alpha^{e}g_{2}x^{2}+\cdots+a_{n}\alpha^{en}g_{n}x^{n})=(g_{1}+vg_{2})f.
\end{align}
That is, $f\in Z(R)$. Therefore, $Z(R)=\mathbb{F}_{p}[x^{p}]$ as
claimed. 
\end{proof}

\begin{corollary} 
$x^{n}-1\in Z(R)$ if and only if $e|n$. 
\end{corollary}

\begin{proof} 
Since $1\in Z(R)$, $x^{n}-1\in Z(R)$ if and only
if $x^{n}\in Z(R)$ if and only if $e|n$. 
\end{proof}

The left and the right division algorithm hold for some elements of
$R$.

\begin{theorem}\label{taghsim} 
Let $f,g\in R$ such that the leading
coefficent of $g$ is a unit. Then there exist unique polynomials
$q$ and $r$ in $R$ such that $f=qg+r$, where $r=0$ or $deg(r)<deg(g)$.
\end{theorem}

\begin{proof}
 The proof is straightforward. 
 \end{proof}

Now we shall determine, $U(R)$, the set of all unit elements of $R$.
First we shall prove the following lemma, which is crucial in over
studies later on

\begin{lemma}\label{lem 2.1} 
For any element $g(x)\in R$, there
exists $g'(x)\in\mathbb{F}_{p}[x]$ such that $vg=g'v$. 
\end{lemma}

\begin{proof} 
Let $g(x)=\sum_{i=0}^{n}g_{i}x^{i}\in R$. Since $g_{i}\in S$
for each $i$, there exist $g'_{i}$ and $g''_{i}$ in $\mathbb{F}_{p}$
such that $g_{i}=g'_{i}+vg''_{i}$. So $g(x)=\sum(g'_{i}+vg''_{i})x^{i}$.
Since $v^{2}=0$, we have 
\begin{align}
vg(x)=\sum v(g'_{i}+vg''_{i})x^{i}=\sum vg'_{i}x^{i}=\sum_{e\nmid i}\alpha^{-i}g'_{i}x^{i}v+\sum_{e|i}g'_{i}x^{i}v=g'(x)v
\end{align}
for some $g'(x)\in R$. 
\end{proof}

\textbf{Notation}. For a fixed element $g\in R$, the element $g'\in\mathbb{F}_{p}[x]$
in lemma \ref{lem 2.1} is unique and hence we call it the \textit{partaker}
of $g$. 

\begin{lemma}\label{lem 4.4} 
Let $A\unlhd\mathbb{F}_{p}[x]$
and $A'\subseteq\mathbb{F}_{p}[x]$, such that $vA=A'v$. Then $A'\unlhd\mathbb{F}_{p}[x]$.
\end{lemma}

\begin{proof} Let $f,g\in A'$ and $h\in R$. So there exist polynomials
$l,k\in A$ such that $fv=vk$ and $gv=vl$. So $(f+g)v=v(l+k)\in vA=A'v$.
Hence $f+g\in A'$. Also $hfv=hvk=vh'k$, for some $h'\in\mathbb{F}_{p}[x]$.
Since $h'k\in A$, $hf\in A'$. Thus $A'$ is an ideal of $R$. \end{proof}
\begin{note} From now on, we shall call $A'$ in lemma \ref{lem 4.4},
the \textit{partaker set} of $A$. \end{note}

First we shall find $U(S)$.

\begin{lemma} $U(S)=\mathbb{F}_{p}^{*}+v\mathbb{F}_{p}$. \end{lemma}

\begin{proof} Let $a+bv\in S$ be a unit. So there exists $c+dv\in S$
such that $(a+bv)(c+dv)=1$. So $ac=1$ and $ad+bc=0$. One can show
that these equations have unique solutions for $c$ and $d$, if and
only if $a\in\mathbb{F}_{p}^{*}$. \end{proof} 
\begin{theorem}\label{thm 2.3} $U(R)=\{a+vh(x)|a\in\mathbb{F}_{p}^{*},h\in\mathbb{F}_{p}[x]\}$
\end{theorem}

\begin{proof} Let $h(x)=\sum_{i=0}^{n}h_{i}x^{i}\in\mathbb{F}_{p}[x]$.
Write $h=b+g(x)$, where $b=h_{0}$ and $g(x)\in\mathbb{F}_{p}[x]$.
We show that $a+vh(x)$, where $a\in\mathbb{F}_{p}^{*}$ has the inverse
$t=(a+bv)^{-1}-(a+bv)^{-1}vg(a+bv)^{-1}$. 
\begin{align}
 & ((a+bv)+vg)[(a+bv)^{-1}-(a+bv)^{-1}vg(a+bv)^{-1}]\nonumber \\
 & =1-vg(a+bv)^{-1}+vg(a+bv)^{-1}-vg(a+bv)^{-1}vg(a+bv)^{-1}\nonumber \\
 & =1-vg(a+bv)^{-1}vg(a+bv)^{-1}=1-v^{2}k(x)=1
\end{align}
for some $k(x)\in\mathbb{F}_{p}[x]$.

Similarly, $t$ is the left inverse of $a+vh(x)$. Thus $a+vh$ is
a unit in $R$. Conversly, let $f\in U(R)$. Then there exists $g\in R$
such that $fg=gf=1$. Let $f=f_{1}+vf_{2}$ and $g=g_{1}+vg_{2}$
for $f_{i},g_{i}\in\mathbb{F}_{p}[x]$. So $fg=(f_{1}+vf_{2})(g_{1}+vg_{2})=1$
implies that $f_{1}g_{1}=1$ and $vf_{2}g_{1}+f_{1}vg_{2}=0$. Hence
$f_{1}$ is a non-zero constant polynomial. That is, $f_{1}\in\mathbb{F}_{p}^{*}$.
Thus $f=f_{1}+vf_{2}$, where $f_{1}\in\mathbb{F}_{p}^{*}$ and $f_{2}\in\mathbb{F}_{p}[x]$.
\end{proof}

\subsection{The left maximal and prime ideals of $R$}

In this section, we shall determine the sets $Max(R)$ and $Spec(R)$,
the set of all left maximal and prime ideals of $R$ respectively.
For the sake of semplicity, from now on, by an ideal of $R$ we mean
a left ideal of $R$.

First, we shall show that $v$ is irreducible in $R$.

\begin{lemma} 
$Rv$ is a maximal ideal in $R$. 
\end{lemma}

\begin{proof} Let $v=fg$, for some $f,g\in R$. Let $f=f_{1}+vf_{2}$
and $g=g_{1}+vg_{2}$. Then $f_{1}g_{1}=0$ and 
\begin{align}
f_{1}vg_{2}+vf_{2}g_{1}=v.\label{yek}
\end{align}
From $f_{1}g_{1}=0$, we have that $f_{1}=0$ or $g_{1}=0$. If $f_{1}=0$,
then $vf_{2}g_{1}=v$. So $f_{2}g_{1}=1$. Hence $g$ is a unit in
$\mathbb{F}_{p}[x]$. If $g_{1}=0$, then by equation (\ref{yek}),
$f_{1}vg_{2}=v$. Let $g'_{2}$ be the partaker if $g_{2}$. Thus
$fg'_{2}v=v$ so $f_{1}g'_{2}=1$. This implies that $f=f_{1}+vf_{2}$
is a unit by theorem \ref{thm 2.3}. Therefore, $Rv$ is a maximal
ideal in $R$. \end{proof}

Now, to determine the sets $Max(R)$ and $Spec(R)$, we shall introduce
the following sets, which in fact are ideals of $\mathbb{F}_{p}[x]$.

\begin{definition} Let $A\unlhd R$. Define 
\begin{align*}
 & A_{[1]}=\{f\in\mathbb{F}_{p}[x]|\exists g\in\mathbb{F}_{p}[x]\quad\textit{such that}\quad f+vg\in A\}\\
 & A_{[2]}=\{g\in\mathbb{F}_{p}[x]|\exists f\in\mathbb{F}_{p}[x]\quad\textit{such that}\quad f+vg\in A\}
\end{align*}
Let $f\in A$ and $f=f_{1}+vf_{2}$, for some $f_{1},f_{2}\in\mathbb{F}_{p}[x]$.
Since $f_{1}\in A_{[1]}$ and $f_{2}\in A_{[2]}$, we conclude that
$A\subset A_{[1]}+vA_{[2]}$. \end{definition}

\begin{lemma}\label{lem 3.3} Let $A\unlhd R$. Then $A_{[1]}$ and
$A_{[2]}$ are ideals of $\mathbb{F}_{p}[x]$. \end{lemma}

\begin{proof} Let $f_{1},f_{2}\in A_{[1]}$. Then there exist $g_{1},g_{2}\in\mathbb{F}_{p}[x]$
such that $f_{1}+vg_{1},f_{2}+vg_{2}\in A$. Thus $f+g=(f_{1}+f_{2})+v(g_{1}+g_{2})\in A$
since $A\unlhd R$. Hence $f_{1}+f_{2}\in A_{[1]}$. Now for $f\in A_{[1]}$
and $g\in\mathbb{F}_{p}[x]$, we show that $fg$ is an element of
$A_{[1]}$. There exists $h\in\mathbb{F}_{p}[x]$ such that $f_{1}+vh\in A$.
So $g(f+vh)\in A$. So $gh\in A_{[1]}$ and hence $A_{[1]}\unlhd\mathbb{F}_{p}[x]$.
Similarly, $A_{[2]}\unlhd\mathbb{F}_{p}[x]$. \end{proof}

\begin{lemma}\label{lem 3.4} For $A\unlhd R$, we have

i) $vA_{[1]}\unlhd R$.

ii) $vA_{[1]}\subseteq A$.

iii) $A_{[1]}\subseteq A_{[2]}$.

iv) If $A_{[1]}=A_{[2]}=\mathbb{F}_{p}[x]$, then $A=R$. \end{lemma}

\begin{proof} i) By lemma \ref{lem 3.3}, $A_{1}$ is an ideal of
$\mathbb{F}_{p}[x]$ and since $\mathbb{F}_{p}[x]$ is a PID, there
exists $a\in\mathbb{F}_{p}[x]$ such that $A_{[1]}=\langle a\rangle$.

We only show that if $f\in R$ and $vg\in vA_{[1]}$, then $fvg\in vA_{[1]}$.
Let $f=f_{1}+vf_{2}$ and $vg=vak$ for some $k\in\mathbb{F}_{p}[x]$.
Then $fvg=(f_{1}+vf_{2})vka=vf'_{1}ka$, where $f'_{1}$ is the partaker
of $f_{1}$. Since $f'_{1},k,a\in\mathbb{F}_{p}[x]$, $fvak\in vA_{[1]}$,
which shows that $vA_{[1]}\unlhd R$.

ii) Let $vk_{1}\in vA_{[1]}$. Then there exists $k_{2}\in A_{[2]}$
such that $k_{1}+vk_{2}\in A$. So $vk_{1}=v(k_{1}+vk_{2})\in A$
(since $A\unlhd R$). So $vA_{[1]}\subseteq A$.

iii) Since $\mathbb{F}_{p}[x]$ is a PID, $A_{[1]}=\langle f\rangle$
and $A_{[2]}=\langle g\rangle$ for some $f,g\in\mathbb{F}_{p}[x]$.
Thus there exists $h\in\mathbb{F}_{p}[x]$ such that $g+vh\in A$.
So $v(g+vh)=vg\in A$. Thus $g\in A_{[2]}$. That is, $A_{[1]}\subseteq A_{[2]}$.

iv) Since $1\in A_{[1]}$, $1+vg(x)\in A$ is a unit in $R$ for $g(x)\in A_{[2]}$
by theorem \ref{lem 3.4}. Thus $A=R$. \end{proof}

We showed that $A\subseteq A_{[1]}+vA_{[2]}$. This inclusion can
be strict, as the following example shows.

\begin{example}\label{mesal} Let $A=R(v+x)$. Then $A_{[1]}=<x>$
and $A_{[2]}=<1>$. We claim that $v$ is not in $A$. Otherwise,
let $v=(f_{1}+vf_{2})(x+v)$ for some $f_{1},f_{2}\in\mathbb{F}_{p}[x]$.
So $xf_{1}=0$ which means that $f_{1}=0$. Thus $vf_{2}x=v$. So
$xf_{2}=1$, which is not possible. So there is no such $f=f_{1}+vf_{2}\in R$
such that $v=f(v+x)$, which means that $v$ is not in $A$. However,
$v\in A_{[1]}+vA_{[2]}$. So $A\subsetneqq A_{[1]}+v_{[2]}$. \end{example}

\begin{definition} Let $A\trianglelefteq R$. $A$ is called a first
type ideal of $R$, if $A=A_{[1]}+vA_{[2]}$, and it is called a second
type if $A\subsetneqq A_{[1]}+vA_{[2]}$. \end{definition} 
\begin{example}\label{second}
This example is a generalization of Example \ref{mesal}. We show
that $A=(f+v)R$ is a second type ideal for every $0\neq f\in\mathbb{F}_{p}[x]$
which is not unique. Let $A=A_{[1]}+vA_{[2]}$. Since $f+v\in A$,
$1\in A_{[2]}$. So $v\in A$. That is, $v=(h_{1}+vh_{2})(f+v)$ for
some $h_{1},h_{2}\in\mathbb{F}_{p}[x]$. Hence $fh_{1}=0$, which
means that $h_{1}=0$. So $v=(vh_{2})(f+v)=vfh_{2}$. So $f$ is a
unit in $\mathbb{F}_{p}[x]$ which is a contradiction. Therefore,
$A$ is a second type ideal of $\mathbb{F}_{p}[x]$. \end{example}

In the following example, we propose a second type ideal which is
not principle.

\begin{example} In this example, we give a non principle second type
ideal $A$ which $x^{n}-1\in A$ for some $n\in\mathbb{N}$. Note
that these ideals are so applicable in encoding and decoding which
we discuss later.

Consider $A=R((x^{3}-1)+v)+v\mathbb{F}_{p}[x](x^{n}-1)$. In this
ideal, $A_{[1]}=<x^{3}-1>$ and $A_{[2]}=\mathbb{F}_{p}[x]$. First,
we show that $A$ is not first ype ideal. Suppose in contrary, $A$
is a first type ideal. Then $v\in A$ which means that there are $f,g,h,k\in\mathbb{F}_{p}[x]$
such that 
\begin{align}
(f+vg)(x^{3}-1+v)+(h+vk)(v(x-1))=v
\end{align}
So, $f=0,g(x^{3}-1)+h'(x-1)=1$. Hence, $x-1|1$ and it is impossible.

Now, we show that there is no generator for $A$. Suppose that $A=R(f+vg)$
for some $f,g\in\mathbb{F}_{p}[x]$. We know that $f=x^{3}-1$ (Otherwise,
$A_{[1]}\neq<x^{3}-1>$). Also, $v(x-1)\in A$. So there exists $h,k\in\mathbb{F}_{p}[x]$
such that 
\begin{align}
(h+vk)(x^{3}-1+vg)=h(x^{3}-1)+vk(x^{3}-1)+vh'g=v(x-1)
\end{align}
So $h=0$ and therefor the left side is equal to $vk(x^{3}-1)$ for
some $k\in\mathbb{F}_{p}[x]$ which is not equal to $v(x-1)$. This
contradiction complete the example. \end{example}

At first, It may be disappointing that the second type ideals are
not principle, but they also have a good property as follows. Its
proof is inspired from \cite{jin}.

\begin{theorem}\label{secondd} Let $A$ be a second type ideal of
$R$. Suppose that $f$ is the polynomial which has the minimum degree.
Then only one of the following occures.

i) The leading coefficient of $f$ is unit and $A=Rf$.


ii) $f=vf_{2}$ and $A=v\mathbb{F}_{p}[x]f_{2}+Rg$ where $g\in A$
is the polynomial with minimum degree such that its leading coefficient
is unit. \end{theorem} \begin{proof} Let $f=\sum_{i=0}^{m}f_{i}x^{i}$
be the polynomial with minimal degree. If $f_{m}$ is not unit, then
according to \ref{taghsim}, for each $h\in A$, there exists unique
factors like $q,r\in R$ such that $h=qf+r$ and $\deg(r)<\deg(f)$.
Since $h-qf\in A$, $r\in A$ which means that $r=0$ by the minimality
of degree of $f$. So $h=qf$ and (i) follows.

Assume that $f_{m}$ is not unit and for each $h\in\mathbb{F}_{p}[x]$,
$f\neq vh$. So $f_{m}=vf'_{m}$. If $j=\min_{i}\{i|f_{i}\notin U(\mathbb{F}_{p}+v\mathbb{F}_{p})\}$,
then $0\neq vf\in A$ and $\deg(vf)<\deg(f)$. This contradiction
results in $f=vf_{2}$ and $A_{[2]}=<f_{2}>$.

If there does not exist $g\in A$ such that $g$ is the polynomial
with minimum degree such that its leading coefficient is unit, then
each polynomial in $A$ is of the form $vh$. So $A=vA_{[2]}$ which
is a first type ideal and contradicts with our assumption. Otherwise,
let $g$ be the polynomial with minimum degree such that its leading
coefficient is unit. Let $\Gamma=\{y\in A|\deg(f)\leq\deg(y)<\deg(g),y\notin Rf\}$.

We claim that $\Gamma=\emptyset$. Otherwise, assume $\Gamma\neq\emptyset$
and $w=w_{1}+vw_{2}\in\Gamma$ which has the minimum degree. We claim
that $w_{1}\neq0$. Otherwise, $w=vw_{2}$ and since $f_{2}h=w_{2}$
($w_{2}\in A_{[2]}$) for some $h\in\mathbb{F}_{p}[x]$, $w=hf$ which
is impossible by the definition of $w$. So $w_{1}\neq0$. Hence,
$0\neq vw\in A$. Since the leading coefficient of $w$ is not unit,
$\deg(vw_{1})<\deg(w)$. So $vw=vw_{1}=hf$ for some $h\in\mathbb{F}_{p}[x]$.
Thus $w_{1}=h'f_{2}$. Also $w_{2}\in A_{[2]}$ which means that $w=w_{1}+vw_{2}=h'f_{2}+vlf_{2}=(h'+vl)f_{2}$.
We know that $\deg(w)<\deg(g)$, so $\deg(w_{1})<\deg(g)$. Also $w_{1}=w-vlf_{2}\in A$.
So $w_{1}$ is a polynomial with degree less than $\deg(g)$ which
is in $A$ and its leading coefficient is unit. This case also is
impossible which results in $\Gamma=\emptyset$.

Let $h\in A$, then there exists $q,r$ such that $h=qg+r$ and $\deg(r)<\deg(g)$.
Since $h-qg\in A$, so $r=f$ or $r=0$. So each $h\in A$ can be
written as the form $qg+f$ or $qg$ for some $q\in R$. So $A=Rf+Rg$.
\end{proof} According to the above theorem, it is easy to see that
sum of a first type ideal and a second type ideal will be a second
type ideal.

Now we are in a position to give a chacterization of maximal ideals
of $R$. \begin{theorem}\label{thm 3.8} Let $A\unlhd R$. Then $A$
is a maximal ideal of $R$ if and only if

i) $A_{[1]}=<f>$, for some irreducible polynomial $f\in\mathbb{F}_{p}[x]$.

ii) $A_{[2]}=\mathbb{F}_{p}[x]$.

iii) $A$ is of the first type, that is, $A=A_{[1]}+v\mathbb{F}_{p}[x]$.
\end{theorem} \begin{proof} $\Rightarrow$) i) We show that $A_{[1]}$
is maximal in $\mathbb{F}_{p}[x]$. Let $A_{[1]}\subsetneqq B\subsetneqq\mathbb{F}_{p}[x]$.
Then $A\subseteq A_{[1]}+v\mathbb{F}_{p}[x]\subsetneqq B+v\mathbb{F}_{p}[x]\subsetneqq R$,
which is a contradiction. So $A_{[1]}$ is maximal in $\mathbb{F}_{p}[x]$
and hence is generated by an irreducible polynomial $f\in\mathbb{F}_{p}[x]$.

ii) Suppose that $A_{[2]}\subsetneqq\mathbb{F}_{p}[x]$. We know that
$A_{[1]}$ is maximal in $\mathbb{F}_{p}[x]$. Let $B=A_{[1]}+v\mathbb{F}_{p}[x]$.
Then $A\subsetneqq B\subsetneqq R$ (since $v\in B-A$ and $1\in\mathbb{F}_{p}[x]-B$),
which is a contradiction.

iii) Let $A\neq A_{[1]}+vA_{[2]}$. Then $A\subsetneqq A_{[1]}+v\mathbb{F}_{p}[x]\subsetneqq R$,
which is a contradiction.

$\Leftarrow$) Since $A$ is a proper ideal, there exists a maximal
ideal $B$ containing $A$. By hypothesis $A_{[2]}=\mathbb{F}_{p}[x]$.
Since $A\subseteq B$, $A_{[2]}=\mathbb{F}_{p}[x]\subseteq B_{[2]}$.
That is, $B_{[2]}=\mathbb{F}_{p}[x]$. Also $A_{[1]}\subseteq B_{[1]}$.
Since $B$ is proper in $R$ , $B_{[1]}\neq\mathbb{F}_{p}[x]$ by
lemma \ref{lem 3.4}$(iv)$. But $A_{[1]}=<f>$, where $f$ is irreducible
in $\mathbb{F}_{p}[x]$ and if $B_{[1]}=<g>$, then $g|f$, which
implies that $f=ug$, for some unit $u\in\mathbb{F}_{p}[x]$. Hence
$B_{[1]}=A_{[1]}$.

Finally $B_{[1]}=A_{[1]}=<f>$, and $B_{[2]}=A_{[2]}=\mathbb{F}_{p}[x]$.
So by Theorem \ref{thm 3.8}, we have $B\subseteq B_{[1]}+vB_{[2]}=A_{[1]}+vA_{[2]}=A$,
as required. \end{proof} Now, we shall find the set of all left prime
ideals of $R$. Note that for any $A\lhd R$, the equation $<v>A=<vA>$
holds. \begin{lemma} Let $Spec(A)$ be the set of all left prime
ideals of $R$. Then $Spec(A)\subseteq Max(R)\cup v\mathbb{F}_{p}[x]$.
\end{lemma} \begin{proof} Let $P$ be a prime ideal of $R$. We
shall show that $P_{[1]}\subseteq P$. We know that $vP_{[1]}$ is
an ideal of $R$ such that $vP_{[1]}\subseteq P$ by lemma \ref{lem 3.4}.
Also we know that $<v><P_{[1]}>\subseteq<vP_{[1]}>\subseteq P$. So
$<v>\subseteq P$ or $<P_{[1]}>\subseteq P$ which means that $v\in P$
or $P_{[1]}\subseteq P$. Assume that $v\in P$. Then let $f\in P_{[1]}$.
So $f+vg\in P$ for some $g\in\mathbb{F}_{p}[x]$. Since $v\in P$,
$f\in P$ which means that $P_{[1]}\subseteq P$.

We show that $P_{[1]}$ is prime in $\mathbb{F}_{p}[x]$. Let $BC\subseteq P_{[1]}$,
for some ideals $B,C$ in $\mathbb{F}_{p}[x]$. Thus $vB(C+vC)\subseteq P$.
So $vB\subseteq P$ or $C+vC\subseteq P$. Hence $vB\subseteq P$
or $C\subseteq P_{[1]}$.

If $vB\subseteq P$, then $(C+vC)(B+vB)=(BC+vBC)+CvB\subseteq P+vP+CP\subseteq P$.
So $B+vB\subseteq P$ or $C+vC\subseteq P$, which implies that $B\subseteq P_{[1]}$
or $C\subseteq P_{[1]}$. So in any case $B\subseteq P_{[1]}$ or
$C\subseteq P_{[1]}$, which means that $P_{[1]}$ is prime in $\mathbb{F}_{p}[x]$.
But $\mathbb{F}_{p}[x]$ is a PID, so $P_{[1]}$ is maximal or the
zero ideal. By lemma \ref{lem 3.4}$(iii)$, $P_{[1]}\subseteq P_{[2]}$.
So we can have three cases.

$i$) $P_{[1]}=P_{[2]}$ maximal in $\mathbb{F}_{p}[x]$.

$ii$) $P_{[2]}=\mathbb{F}_{p}[x]$.

$iii$) $P_{[1]}=0$, that is, $P=vP_{[2]}$.

Suppose that $P_{[1]}=P_{[2]}=\langle\pi\rangle$ for some irreducible
polynomial $\pi\in\mathbb{F}_{p}[x]$. Let $k\in\mathbb{F}_{p}[x]$
be an irreducible polynomial in $\mathbb{F}_{p}[x]$ which is different
from $\pi$. Then $(0+vk\mathbb{F}_{p}[x])(\pi\mathbb{F}_{p}[x]+v\mathbb{F}_{p}[x])\subseteq vk\pi\mathbb{F}_{p}[x]\subseteq v\pi\mathbb{F}_{p}[x]\subseteq P$.
But neither $vk\mathbb{F}_{p}[x]\subseteq P$ nor $\pi\mathbb{F}_{p}[x]+v\mathbb{F}_{p}[x]$,
since $k\neq\pi$. So $P$ is not prime in this case.

Suppose that $P_{[2]}=\mathbb{F}_{p}[x]$ and $P_{[1]}=\langle\pi\rangle$
for some irreducible $\pi\in\mathbb{F}_{p}[x]$. So $\pi+vs\in P$
for some $s\in\mathbb{F}_{p}[x]$. Thus $v(\pi+vs)=v\pi\in P$. We
show that $<v><\pi+v>\subseteq P$. Let $a\in<v><\pi+v>$. Then 
\begin{align}
a=\sum_{i}(f_{i}+vg_{i})v(k_{i}+vl_{i})(\pi+v)=\sum f_{i}vk_{i}\pi\in<v\pi>\subseteq P.
\end{align}
Hence $<v>\subseteq P$ or $<\pi+v>\subseteq P$. If $v\in P$, then
by Theorem \ref{thm 3.8}, $P$ is maximal, since $\pi\in P_{[1]}\subseteq P$
and hence $P=\pi\mathbb{F}_{p}[x]+v\mathbb{F}_{p}[x]=P_{[1]}+vP_{[2]}$.
But if $\pi+v\in P$, then $v\pi=v(\pi+v)\in P$ and so $<v\pi>\subseteq P$.
Since $<v><\pi>\subseteq<v\pi>\subseteq P$, $<v>\subseteq P$ or
$<\pi>\subseteq P$, in each case $P$ is maximal.

Finally, let $P_{[1]}=0$. Then $P=vP_{[2]}$. Thus $P=v\rho\mathbb{F}_{p}[x]$,
for some non-zero $\rho\in\mathbb{F}_{p}[x]$. We show that $\rho$
is a unit. $P$ can not be zero, as $v^{2}=0\in P$, but $v\notin P$.
So 
\begin{align}
(v\mathbb{F}_{p}[x])(\rho\mathbb{F}_{p}[x]+v\rho\mathbb{F}_{p}[x])\subseteq v\rho\mathbb{F}_{p}[x]=P.
\end{align}
Thus $(\rho\mathbb{F}_{p}[x]+v\rho\mathbb{F}_{p}[x])\subseteq P$
or $v\mathbb{F}_{p}[x]\subseteq P$. So $\rho+vh\in P$ for some $h\in\mathbb{F}_{p}[x]$
or $v\in P$. Since $\rho\in P_{[1]}=0$, which in impossible as $\rho\neq0$.
Hence $v\in P$. So $1\in P_{[2]}$ (since $0+v.1\in P$) and hence
$P_{[2]}=\mathbb{F}_{p}[x]=\rho\mathbb{F}_{p}[x]$. That is, $\rho$
is a unit. Therefore, $P=v\mathbb{F}_{p}[x]$ is required in this
case. \end{proof}


\begin{proof} By Theorem \ref{thm 3.8}, $M=P+v\mathbb{F}_{p}[x]$
for some maximal (and hence prime) ideal $P$ of $\mathbb{F}_{p}[x]$.
Let $AB\subseteq M$. If $f\in(AB)_{[1]}$, then $f+vg\in AB\subseteq M$
for some $g\in\mathbb{F}_{p}[x]$. So $f\in P$. Hence $f=\sum_{i<\infty}a_{i}b_{i}\in P$,
for some $a_{i}\in A_{[1]}$ and $b_{i}\in B_{[1]}$. Thus $A_{[1]}B_{[1]}\subseteq P$.
But $P$ is prime in $\mathbb{F}_{p}[x]$, so $A_{[1]}\subseteq P$
or $B_{[1]}\subseteq P$. Thus $A\subseteq A_{[1]}+vA_{[2]}\subseteq M$
or $B\subseteq B_{[1]}+vB_{[2]}\subseteq M$. Hence $A\subseteq M$
or $B\subseteq M$. So $M$ is prime. \end{proof}

\begin{lemma}\label{lem 3.12} The ideal $P=v\mathbb{F}_{p}[x]$
is prime in $R$. \end{lemma}

\begin{proof} Suppose that $AB\subseteq P$ for $A,B\unlhd R$. So
$A_{[1]}B_{[1]}\subseteq P_{[1]}=0$. Hence, $A_{[1]}=0$ or $B_{[1]}=0$.
Thus $A=vA_{[2]}$ or $B=vB_{[2]}$. Therefore, $A\subseteq P$ or
$B\subseteq P$. \end{proof} Now by the above results we can give
a characterisation of all left prime ideals of $R$.

\begin{theorem}\label{thm 3.13} $Spec(R)=Max(R)\cup v(\mathbb{F}_{p}[x])$.
\end{theorem}

\subsection{The primary ideals of $R$}

In this section, we shall give some characterisations of left primary
ideals of $R$. Recall that a left (respectively, right) proper ideal
$Q$ is called primary if for each left(respectively, right) ideals
$A$ and $B$ such that $AB\subseteq Q$, then $A\subseteq Q$ or
there exists $n\in\mathbb{N}$ such that $B^{n}\subseteq Q$. (respectively,
$B\subseteq Q$ or $A^{n}\subseteq Q$ for some $n\in\mathbb{N}$),
see, for example, \cite{chatters}. First, we shall show that the
irreducible polynomials of $\mathbb{F}_{p}[x]$ are irreducible in
$R$.

\begin{lemma} An element $f\in\mathbb{F}_{p}[x]$ is irreducible
in $R$ if and only if $f$ is irreducible in $\mathbb{F}_{p}[x]$.
\end{lemma}

\begin{proof} $\Rightarrow$) Let $f\in\mathbb{F}_{p}[x]$ be irreducible
in $\mathbb{F}_{p}[x]$, but $f=gh$, for some $g,h\in R$. Let $g=g_{1}+vg_{2}$
and $h=h_{1}+vh_{2}$, where $f_{i},g_{i},h_{i}\in\mathbb{F}_{p}[x]$
for $i=1,2$. Then $f_{1}=g_{1}h_{1}$. So $g_{1}$ or $h_{1}$ is
a unit in $\mathbb{F}_{p}[x]$ and hence $g$ or $h$ is a unit by
Theorem \ref{thm 2.3}.

$\Leftarrow$) Obvious. \end{proof}

\begin{note}\label{note1} Recall that if $R$ is a $UFD$ and $\pi$
is an irreducible element of $R$, then $<\pi>$ is a prime ideal
and $<\pi^{n}>$ ,$n\ge1$, is a primary ideal, with radical $<\pi>$.
Conversely, every primary ideal $Q$ whose radical is $<\pi>$ is
of the form $<\pi^{n}>$ , $n\ge1$. (see, for example, \cite{zariski},
P.155.) \end{note} 
\begin{lemma}\label{lem 4.3} Let $S$ be a PID. Then $Q\vartriangleleft S$
is primary if and only if for each ideals $B,C\trianglelefteq S$,
if $BC\subseteq Q$, then $B\subseteq\sqrt{Q}$ or $C\subseteq\sqrt{Q}$.
\end{lemma}

\begin{proof} $\Rightarrow$) Obvious.

$\Leftarrow$) Let $Q=\prod_{i=1}^{k}P_{i}^{a_{i}}$, the prime factorization
of $Q$ into prime ideals of $S$. If $k>1$, then $\prod_{i=1}^{k}P_{i}^{a_{i}}\subseteq Q$,
but neither $P_{1}^{a_{1}}\nsubseteq\prod_{i=1}^{k}P_{i}=\sqrt{Q}$
nor $\prod_{i=2}^{k}P_{i}^{a_{i}}\nsubseteq\prod_{i=1}^{k}P_{i}=\sqrt{Q}$,
which is a contradiction with hypothosis. So $k=1$ and hence $Q$
is primary, since $Q$ is a power of a maximal ideal in $S$. \end{proof}

\begin{theorem}\label{thm 4.5} Let $Q$ be a left primary ideal
of $R$. Then only one of the following cases occures.

$i$) There exists a prime ideal $P\unlhd\mathbb{F}_{p}[x]$ with
partaker $P'$, where $P'\subseteq P$ and positive integers $a,b$
such that $a\geq b$, $Q_{[1]}=P^{a}$ and $Q_{[2]}=P^{b}$.

$ii$) There exists a prime ideal $P\unlhd\mathbb{F}_{p}[x]$ and
integer $a>0$ such that $Q_{[1]}=P^{a}$ and $Q_{[2]}=\mathbb{F}_{p}[x]$.

$iii$) $Q=v\mathbb{F}_{p}[x]$

$iv$) $Q=0$. \end{theorem}

\begin{proof} First, we show that $Q_{[1]}$ is a primary ideal of
$\mathbb{F}_{p}[x]$. Note that $Q_{[1]}$ is a proper ideal of $Q$
since otherwise, $Q=R$ by lemma \ref{lem 3.4}$(iv)$.

Let $BC\subseteq Q_{[1]}$ for $B,C\unlhd\mathbb{F}_{p}[x]$. One
can see that $vB(C+vC)\subseteq Q$. So $vB\subseteq Q$ or $(C+vC)^{n}\subseteq Q$
for some $n\in\mathbb{N}$, since $Q$ is primary and $vB,C+vC\unlhd R$.
If $(C+vC)^{n}\subseteq Q$ for some $n\in\mathbb{N}$, then $C^{n}\in Q_{[1]}$.
If $vB\subseteq Q$, then $<v><B>\subseteq\langle vB\rangle\subseteq Q$.
So $v\in Q$ or $<B>^{m}\subseteq Q$ for some $m\in\mathbb{N}$,
which implies that $B^{m}\subseteq Q_{[1]}$. So we have proved that
$v\in Q$ or $B^{n}\subseteq Q_{[1]}$ or $C^{m}\subseteq Q_{[1]}$.
Now, we show that $v\in Q$ leads to $B\subseteq\sqrt{Q_{[1]}}$ or
$C\subseteq\sqrt{Q_{[1]}}$. Let $v\in Q$. Suppose that $k\in BC\subseteq Q_{[1]}$.
So there exist $k'\in Q_{[2]}$ such that $k+vk'\in Q$. Since $vk'\in Q$,
$k\in Q$. That is, $BC\subseteq Q$. Now $(B+vB)(C+vC)\subseteq Q$.
Hence $B+vB\subseteq Q$ or $(C+vC)^{n}\subseteq Q$ for some $n\in\mathbb{N}$.
Thus $B\subseteq Q_{[1]}\subseteq\sqrt{Q_{[1]}}$ or $C\subseteq\sqrt{Q_{[1]}}$
as required.

Therefore, we have shown that if $BC\subseteq Q_{[1]}$ for $B,C\unlhd\mathbb{F}_{p}[x]$,
then $B\subseteq\sqrt{Q_{[1]}}$ or $C\subseteq\sqrt{Q_{[1]}}$. Since
$\mathbb{F}_{p}[x]$ is a PID, lemma \ref{lem 4.3} shows that $Q_{[1]}$
is a primary ideal of $\mathbb{F}_{p}[x]$. Hence $\sqrt{Q_{[1]}}=P$
is a prime and hence maximal or zero by Theorem \ref{thm 3.13}. Thus
$Q_{[1]}=P^{a}$ or $Q_{[1]}=0$ for some positive integer $a$ and
a non-zero prime ideal $P\unlhd\mathbb{F}_{p}[x]$, by Note \ref{note1}.

First, suppose that $Q_{[1]}=P^{a}$. Since $Q_{[1]}\subseteq Q_{[2]}$,
so $Q_{[2]}=P^{b}$ for some $b\le a$. If $b=0$, then $(ii)$ does
hold. Let $b>0$. We show that if $P'+P=\mathbb{F}_{p}[x]$, where
$P'$ is the partaker of $P$, then $Q$ is not primary. So let $P+P'=\mathbb{F}_{p}[x]$,
and $P=\langle\pi\rangle$, where $\pi$ is an irreducible polynomial
in $\mathbb{F}_{p}[x]$. Let $k\neq\pi$ be another irreducible polynomial
in $\mathbb{F}_{p}[x]$. Then 
\begin{align}
(vk\mathbb{F}_{p}[x])(\pi^{a}\mathbb{F}_{p}[x]+v\mathbb{F}_{p}[x])=vk\pi^{a}\mathbb{F}_{p}[x]\subseteq v\pi^{a}\mathbb{F}_{p}[x]\subseteq vP^{a}=vQ_{[1]}\subseteq Q.
\end{align}
However, $vk\mathbb{F}_{p}[x]\nsubseteq Q$. So let $(\pi^{a}\mathbb{F}_{p}[x]+v\mathbb{F}_{p}[x])^{r}\subseteq Q$
for some $r\in\mathbb{N}$. Hence $(\pi^{a}+v)^{r}\in Q$ for some
$r\in\mathbb{N}$. Thus 
\begin{align}
(\pi^{a}+v)^{r}=\pi^{ar}+v\pi^{a(r-1)}+v\pi^{a(r-2)}(\pi')^{a}+\cdots+v(\pi')^{a(r-1)}\in Q
\end{align}
where $\pi'$ is the partaker of $\pi$. So $(\pi')^{a(r-1)}\in P^{b}$,
since $v\pi^{a(r-i)}(\pi')^{ai}\in P^{b}=Q_{[1]}$. But $P$ and $P'$
are coprime, and hence $(\pi',\pi)=1$. So $\pi'^{a(r-1)}\in P^{b}$
would be impossible. Hence $P'\subseteq P$ (otherwise, $P+P'=\mathbb{F}_{p}[x]$
since $P$ is maximal in $\mathbb{F}_{p}[x]$).

Now, suppose that $Q_{[1]}=0$. Then $Q=Q_{[1]}+vQ_{[2]}=vQ_{[2]}$.
So 
\begin{align}
vQ_{[2]}=v\mathbb{F}_{p}[x](Q_{[2]}+vQ_{[2]})\subseteq vQ_{[2]}=Q.
\end{align}
Hence $v\mathbb{F}_{p}[x]\subseteq Q$ or $(Q_{[2]}+vQ_{[2]})^{n}\subseteq Q$
for some $n\in\mathbb{N}$. Thus $\mathbb{F}_{p}[x]\subseteq Q_{[2]}$
or $(Q_{[2]}+vQ_{[2]})^{n}\subseteq Q$. If $\mathbb{F}_{p}[x]\subseteq Q_{[2]}$,
then $Q_{[2]}=\mathbb{F}_{p}[x]$ and $(iii)$ is satisfied. However,
if $(Q_{[2]}+vQ_{[2]})^{n}\subseteq Q$, then $Q_{[2]}^{n}\subseteq Q_{[1]}=0$.
So $Q_{[2]}=0$. Therefore, $Q=0$ and $(iv)$ is satisfied. \end{proof}
\begin{lemma}\label{eshterak} For $A,B\unlhd R$, we have 
\begin{align*}
A\cap B=(A_{[1]}\cap B_{[1]})+v(A_{[2]}\cap B_{[2]}).
\end{align*}
In particular, the intersection of two first type ideals is again
a first type ideal. \end{lemma}

\begin{proof} Let $A_{[1]}=a_{1}\mathbb{F}_{p}[x]$, $A_{[2]}=a_{2}\mathbb{F}_{p}[x]$,
$B_{[1]}=b_{1}\mathbb{F}_{p}[x]$ and $B_{[2]}=b_{2}\mathbb{F}_{p}[x]$.
Since $a_{1}\in A$ and $b_{1}\in B$, $lcm(a_{1},b_{1})\in A\cap B$.
Similarly, $lcm(a_{2},b_{2})\in A\cap B$. Thus $lcm(a_{1},b_{1})\mathbb{F}_{p}[x]+v\mathbb{F}_{p}[x]lcm(a_{2},b_{2})\subseteq A\cap B$.
Hence 
\begin{align}
(A_{[1]}\cap B_{[1]})+v(A_{[2]}\cap B_{[2]})\subseteq A\cap B.
\end{align}
Now, let $f\in A\cap B$. If $f=f_{1}+vf_{2}$, for some $f_{1},f_{2}\in\mathbb{F}_{p}[x]$,
then $f_{1}=a_{1}h$, $f_{1}=b_{1}g$, $f_{2}=a_{2}h'$ and $f_{2}=b_{2}g'$.

So $lcm(a_{1},b_{1})|f_{1}$ and $lcm(a_{2},b_{2})|f_{2}$. Hence,
$f_{1}\in A_{[1]}\cap B_{[1]}$ and $f_{2}\in A_{[2]}\cap B_{[2]}$.
So $f\in(A_{[1]}\cap B_{[1]})+v(A_{[2]}\cap B_{[2]})$, as required.
\end{proof}

Now, we prove the converse of Theorem \ref{thm 4.5}.

\begin{theorem} Any proper first type ideal $Q$ of $R$ which satisfies
each one of the following cases, is primary.

$i$) $Q=P^{a}+v\mathbb{F}_{p}[x]$ for some prime ideal $P$ of $\mathbb{F}_{p}[x]$
and some positive integer $a$.

$ii$) $Q=P^{a}+vP^{b}$ for some non-zero prime ideal $P$ which
contains its partiner $P'$ and for some positive integer $a,b$ such
that $a>b$.

$iii$) $Q=v\mathbb{F}_{p}[x]$

$iv$) $Q=0$. \end{theorem} \begin{proof} $i$) Suppose that $Q=P^{a}+v\mathbb{F}_{p}[x]$,
where $P$ is a prime ideal of $\mathbb{F}_{p}[x]$ and $a\in\mathbb{N}$.
Let $BC\subseteq Q$. Then $B_{[1]}C_{[1]}\subseteq Q_{[1]}=P^{a}$.
Since $P^{a}$ is primary, $B_{[1]}\subseteq P^{a}$ or $C_{[1]}\subseteq\sqrt{P^{a}}=P$.
If $B_{[1]}\subseteq P^{a}$, then 
\begin{align}
B\subseteq B_{[1]}+vB_{[2]}\subseteq P^{a}+v\mathbb{F}_{p}[x]=Q.
\end{align}
However, if $C_{[1]}\subseteq P$, we show that $(P+v\mathbb{F}_{p}[x])^{a}\subseteq Q$
for some $a\in\mathbb{N}$. Let $f_{i}\in P+v\mathbb{F}_{p}[x]$ for
$i\leq a$. Then $\prod_{i}f_{i}\in P^{a}+v\mathbb{F}_{p}[x]=Q$.
Hence, $(P+v\mathbb{F}_{p}[x])^{a}\subseteq Q$. We assumed that $C_{[1]}\subseteq P$.
So $C^{a}\subseteq(C_{[1]}+vC_{[2]})^{a}\subseteq(P+v\mathbb{F}_{p}[x])^{a}\subseteq Q$.
Thus $Q$ is primary in this case.

$ii$) Let $Q=P^{a}+vP^{b}$ for some prime $P$ which contains its
partaker $P'$ and for some $a,b\in\mathbb{N}$ with $a>b$. Let $AB\subseteq Q=P^{a}+vP^{b}$,
for some $A,B\trianglelefteq R$. So $(AB)_{[1]}=A_{[1]}B_{[1]}\subseteq P^{a}$.
Hence $A_{[1]}\subseteq P^{a}$ or $B_{[1]}\subseteq P$. Suppose
that $B_{[1]}\subseteq P$ and let $y_{i}=b_{1,i}\pi+vb_{2,i}\in B$,
where $b_{1,i}\pi\in B_{[1]}\subseteq P$ and $b_{2,i}\in B_{[2]}$.
\begin{align*}
\prod_{i\leq a}y_{i}=&\prod_{i\leq a}(b_{1,i}\pi+vb_{2,i})\\?
=&\prod_{i\leq a}b_{1,i}\pi^{a}+\prod_{i\leq a}b_{1,i}\pi^{a-1}vb_{2,i}+\prod_{i\leq a}b_{1,i}\pi^{a-2}vb_{1,i}\pi b_{2,i}+\cdots+\prod_{i\leq a}v(b_{1,i}\pi)^{a-1}b_{2,i}
\end{align*}
Since $P'\subseteq P$, there exists $t\in\mathbb{F}_{p}[x]$ such
that $\pi'=t\pi$. So there exists $b'\in\mathbb{F}_{p}[x]$ such
that 
\begin{align*}
\prod_{i\leq a}y_{i}=\prod_{i\leq a}b_{1,i}\pi^{a}+vb't^{a}\pi^{a}\in\pi^{a}\mathbb{F}_{p}[x]+v\pi^{b}\mathbb{F}_{p}[x]=P^{a}+vP^{b}=Q.
\end{align*}
Hence $\prod_{i\leq a}y_{i}\in Q$. That is, $B^{a}\subseteq Q$.

Now, let $A_{[1]}\subseteq P^{a}$. If $A_{[2]}\subseteq P^{b}$,
then $A\subseteq P^{a}+vP^{b}=Q$, which is done. So let $A_{[2]}\nsubseteq P^{2}$.
Hence there exists $s\in\mathbb{N}\cup\{0\}$ such that $s<b$, $\pi^{s}h\in A_{[2]}$
with $gcd(\pi,h)=1$. So there exists $r\in\mathbb{F}_{p}[x]$ such
that $r\pi^{a}+v\pi^{s}h\in A$.

Let $y=b_{1}+vb_{2}\in B$. Then 
\begin{align*}
(r\pi^{a}+v\pi^{s}h)y=r\pi^{a}b_{1}+r\pi^{a}vb_{2}+v\pi^{s}hb_{1}\in AB\subseteq Q=P^{a}+vP^{b}.
\end{align*}
Now, there exists $l\in\mathbb{F}_{p}[x]$ such that $r\pi^{a}vb_{2}=vl\pi^{a}b_{2}$.
Hence 
\begin{align*}
(r\pi^{a}+v\pi^{s}h)y=r\pi^{a}b_{1}+vl\pi^{a}b_{2}+v\pi^{s}hb_{1}\in P^{a}+vP^{b}.
\end{align*}
So $v\pi^{s}hb_{1}\in vP^{b}$, which does hold provided that $b_{1}\in P$.
But, this implies that $B_{[1]}\subseteq P$, which we get that $B^{n}\subseteq Q$
for some $n\in\mathbb{N}$, as required.

Let $Q=v\mathbb{F}_{p}[x]$, then by lemma \ref{lem 3.12}, $Q$ is
prime, and hence primary.

Let $Q=0$. suppose that $AB\subseteq Q$. Then $A_{[1]}B_{[1]}\subseteq0$.
Hence $A_{[1]}=0$ or $B_{[1]}=0$. Thus $A=vA_{[2]}$ or $B=vB_{[2]}$.
Let $A=vA_{[2]}$. Then $vA_{[2]}B_{[1]}=AB\subseteq0$. So $A_{[2]}=0$
or $B_{[1]}=0$. If $A_{[2]}=0$, then $A=0$. Else, $B=vB_{[2]}$.
So $B^{2}=vB_{[2]}vB_{[2]}=0$. \end{proof}

\begin{lemma} Let $A\unlhd R$ and $A$ is second type. Then, there
exists a first type ideal $B\unlhd R$ such that $B\subseteq A$.
In particular, if $A_{[1]}=\langle f\rangle$, $A_{[2]}=\langle g\rangle$,
then $\langle f\hat{f}\rangle+v\langle f\rangle\subseteq A$(partaker
of $f$ is $\hat{f}$). \end{lemma}

\begin{proof} Let $f+vh\in A$ for some $s\in\mathbb{F}_{p}[x]$.
So $vf\in A$. Also, $\hat{f}(f+vh)\in A$. Hence, $\hat{f}f+vfh\in A$
and this results in $f\hat{f}\in A$. So $\langle f\hat{f}\rangle+v\langle f\rangle\subseteq A$.
\end{proof}

\begin{lemma} Let $P$ be a prime ideal of $R$. Then, $(\pi,\hat{\pi})=1$.
\end{lemma}

\begin{proof} Let $\pi|\hat{\pi}$. Then, $\pi h=\hat{\pi}$. Let
$\pi=\sum_{i=0}^{m}p_{i}x^{i}$, $\hat{\pi}=\sum_{i=0}^{m}\hat{p_{i}}x^{i}$
and $h=\sum_{i=0}^{r}h_{i}x^{i}$. So for $m+1\le r\le r+m$, $\sum_{i=0}^{r}p_{i}\alpha^{i}h_{r-i}=0$.

We have $r$ equations and $r$ undetermined, and equations are independent.
This results in $h=0$. So $\hat{\pi}=0$ which means that $\pi v=v\hat{\pi}=0$.
So $\pi=0$, which is impossible. \end{proof}

\begin{theorem}\label{vali} There is not a second type primary ideal
in $R$. \end{theorem}

\begin{proof} Let $BC\subseteq Q_{[1]}$. One can prove that $B^{m}\subseteq Q$
or $C^{m'}\subseteq Q$ or $v\in Q$ for some $m,m'\in\mathbb{N}$
in similar to thm 4.4.

If $v\in Q$ and $Q_{[1]}=\langle f\rangle$, $f\in Q$ by the fact
that $f+vq\in Q$ for some $g\in Q_{[2]}$. Hence, $Q$ must be of
the first type ideal. So $Q_{[1]}$ is in the form of $P=\langle\pi\rangle^{l}$
for some irreduciblepolynomial in $\mathbb{F}_{p}[x]$. Let $Q_{[2]}=\langle\pi\rangle^{l-s}=P^{l-s}$.
Suppose that $\pi$ be the partaker of $\hat{\pi}$. Also one can
see that 
\begin{align}
(\langle\pi\hat{\pi}\rangle^{l}+v\langle\pi\rangle^{l-s})(\langle\pi\rangle^{s}+v\mathbb{F}_{p})\subseteq\langle\pi^{l-s}\hat{\pi}^{l}\rangle+v\langle\pi\rangle^{l}+v\langle\pi'\pi\rangle^{l}\subseteq\langle\pi\hat{\pi}\rangle^{l}+v\langle\pi\rangle^{l}\subseteq Q.
\end{align}
So $\langle\pi\hat{\pi}\rangle^{l}+v\langle\pi\rangle^{l-s}\subseteq Q$
or $(\langle\pi\rangle^{s}+v\mathbb{F}_{p}[x])^{m}\subseteq Q$. So
$(\pi^{s}+v)^{m}\in Q$ which results in $(\pi')^{s-1}+\pi w\in Q_{[2]}$
for some $w_{1}\in\mathbb{F}_{p}[x]$. Hence, there exists some $k\in\mathbb{F}_{p}[x]$
such that $(\pi')^{s-1}+\pi w_{1}=\pi^{l-s}k$. So $\pi|\pi'$ or
$l-s=0$. $\pi|\pi'$ is not possible by previous lemma. Thus, $Q_{[2]}=\mathbb{F}_{p}[x]$.

Let $L\in\mathbb{F}_{p}[x]$. So for each $c,d\in\mathbb{F}_{p}[x]$,
there exists $w_{2}\in\mathbb{F}_{p}[x]$ such that $(\pi^{l-1}\hat{\pi}^{l}+v\pi')(c+vd)(\pi+vL)=\pi^{l}\hat{\pi}^{l}+v\pi^{l}w_{2}\in Q$.
One can see that $\hat{\pi}^{l}\pi^{l-1}\notin Q_{[1]}$, so $\hat{\pi}^{l}\pi^{l-1}+v\pi^{l}\notin Q$.
So $(\pi+vL)^{m}\in Q$ for some $m\in\mathbb{N}$. Hence, $\pi^{m}+v(\sum_{i=0}^{m-1}\pi^{i}(\pi')^{m-1-i})L\in Q$.
Thus, 
\begin{align}
\pi^{m}+v(\sum_{i=0}^{m-1}\pi^{i}(\hat{\pi'})^{m-1-i})L-\pi^{m}-v(\sum_{i=0}^{m-1}\pi^{i}(\pi')^{m-1-i})=v\sum\pi^{i}(\pi')^{m-i-1}\in Q.
\end{align}
Let $r=min\{\alpha|v\pi^{\beta}\in Q\quad\textsl{for}\quad\beta\ge\alpha\}$.
So $\hat{\pi}^{r-1}v\sum_{i=0}^{m-1}\pi^{i}(\pi')^{m-i-1}\in Q$ which
results in $v\pi^{r-1}(\pi')^{m-1}\in Q$.

Since $(\pi,\pi')=1$, there exists $z_{1},z_{2}\in\mathbb{F}_{p}[x]$
such that $(\pi')^{m-1}z_{1}+\pi z_{2}=1$. We know that $v\pi^{r}\in Q$,
so $\hat{z_{2}}v\pi^{r}\in Q$. Thus, 
\begin{align*}
\hat{z_{1}}+v\pi^{r-1}(\pi')^{m-1}+\hat{z_{2}}v\pi^{r}=v(z_{1}\pi^{r-1}(\pi')^{m-1}+z_{2}\pi^{r})=v\pi^{r-1}(z_{1}(\pi')^{m-1}+\pi z_{2})=v\pi^{r-1}\in Q.
\end{align*}
This is a contradiction by definition of $r$. \end{proof}

According to the above results, we could characterize first type primary
ideals. We will study more about the role of the first and the second
type primary ideals in primary decomposition as follows. First we
prove the following lemma. \begin{theorem}\label{decomposition}
If $A\unlhd R$ is a second type ideal and has a primary decomposition,
then at least one of its components in primary decomposition of $A$
must be of second type. \end{theorem} \begin{proof} 
Let there exists $m$ primary ideals in decomposition of $A$. We
prove the case $m=2$. The general case is followed by induction.
Let $A=Q\cap T$ for some primary ideals $Q$ and $T$ of $R$. If
both of $Q$ and $T$ are of the first type, then by lemma \ref{eshterak}
\begin{align*}
A=Q\cap T=(Q_{[1]}+vQ_{[2]})\cap(T_{[1]}+vT_{[2]})=(Q_{[1]}\cap T_{[1]})+v(Q_{[2]}\cap T_{[2]})
\end{align*}
Which is a first type ideal. So $A$ is a first type ideal which is
a contradiction by assumption. \end{proof} 

\section{Over the ring $\frac{(\mathbb{F}_p+v\mathbb{F}_p)[x;\theta]}{<x^n-1>}$}

\subsection{Correspondence of the ideals of $\frac{R[x]}{<x^{n}-1>}$ and skew
cyclic codes of length $n$ over $R$}

In commutative case, the cyclic codes over $R$ are ideals of $R_{n}=\frac{R[x]}{<x^{n}-1>}$.
In non-commutative case, we prove that the skew cyclic codes over
$S=F_{p}+vF_{p}$ where $v^{2}=0$ are in fact the ideals of $T_{n}=\frac{\frac{F_{p}[x]}{<x^{n}-1>}[v,\theta']}{v^{2}}$
(Assume $O(\theta')\nmid n$). First, we try to find skew cyclic codes
over $R$. The following theorem can introduce skew cyclic codes over
$S$. \begin{theorem} If $O(\theta)|n$, then each skew cyclic code
$\complement$ is an ideal $\overline{A}\unlhd\frac{R}{<x^{n}-1>}$.
\end{theorem} \begin{proof} Let $u\in\complement$ is a codeword,
then 
\begin{align}
f(x)u=\left(\sum_{i=1}^{m}f_{i}x^{i}\right)u
\end{align}
Since $\complement$ is skew cyclic code, then $x^{i}u$ is a codeword,
too. Also $\complement$ is linear code, so $f(x)u=\sum_{i=1}^{m}f_{i}x^{i}u\in\complement$.
So $\complement$ is an ideal in $R_{1}$. Let $f,g\in\overline{A}\unlhd\frac{R}{<x^{n}-1>}$,
then $f-g\in\overline{A}$. Also because $\overline{A}$ is ideal,
$x^{i}f=\sum_{i=1}^{m}\theta^{i}(f_{i})x^{i}\in\overline{A}$. So
both linerity and property of skew cyclic codes are satisfied. Hence
$\overline{A}$ is skew cyclic code. \end{proof} Our goal in this
section is to show the equivalence of ideals of $R_{n}$ (or the skew
cyclic codes over $F_{p}+vF_{p}$) and the ideals of $T_{n}$. In
the first step, we prove that $\theta'$ is well-defined. We know
\begin{align*}
\theta':\frac{F_{p}[x]}{<x^{n}-1>}\longrightarrow\frac{F_{p}[x]}{<x^{n}-1>},\quad\quad\theta'(\overline{1})=\overline{1},\theta'(\overline{x})=\alpha^{-1}\overline{x},\quad\quad\alpha\in F_{p}
\end{align*}
Also, $O(\alpha)=O(\theta')$ (i.e. $O(\alpha)|n$). Let $h,g\in F_{p}[x]$
such that $\overline{h}=\overline{g}$. So $x^{n}-1|h-g$. Moreover,
$\theta'(\overline{h})=\theta'(\sum_{i}\overline{h_{i}}\overline{x}^{i})=\sum_{i}\overline{h_{i}}\theta'(\overline{x})^{i}=\sum_{i}\overline{h_{i}}\alpha^{-i}\overline{x}^{i}$
and $\theta'(\overline{g})=\theta'(\sum_{i}\overline{g_{i}}\overline{x}^{i})=\sum_{i}\overline{g_{i}}\theta'(\overline{x})^{i}=\sum_{i}\overline{g_{i}}\alpha^{-i}\overline{x}^{i}$.

We know $x^{n}-1|\sum_{i}(h_{i}-g_{i})x^{i}$, so $(\alpha^{-1}x)^{n}-1|\sum_{i}(h_{i}-g_{i})(\alpha^{-1}x)^{i}$.
Since $\alpha^{n}=1$, $x^{n}-1|\sum_{i}(h_{i}-g_{i})\alpha^{-i}x^{i}$.
So $\sum_{i}\overline{h_{i}}\alpha^{-i}\overline{x}^{i}=\sum_{i}\overline{g_{i}}\alpha^{-i}\overline{x}^{i}$.
So $\theta'(\overline{h})=\theta'(\overline{g})$. Thus $\theta'$
is well-defined.

Furthemore, $\theta'$ is a ring homomorphism. Suppose that $\overline{f},\overline{g}\in\frac{F_{p}[x]}{<x^{n}-1>}$.
Then, if $f=\sum_{i}\overline{f_{i}}\overline{x}^{i}$, $g=\sum_{i}\overline{g_{i}}\overline{x}^{i}$,
\begin{align*}
\theta'(\overline{f}\overline{g})= & \theta'((\sum_{i}\overline{f_{i}}\overline{x}^{i})(\sum_{i}\overline{g_{i}}\overline{x}^{i}))=\theta'(\sum_{i}\sum_{j}\overline{f_{i}}\overline{g_{i}}\overline{x}^{i+j})\\
= & \sum_{i}\sum_{j}\overline{f_{i}}\overline{g_{i}}\theta'(\overline{x})^{i+j}=\sum_{i}\sum_{j}\overline{f_{i}}\overline{g_{i}}\alpha^{-i-j}\overline{x}^{i+j}\\
= & (\sum_{i}\overline{f_{i}}\alpha^{-i}\overline{x}^{i})(\sum_{i}\overline{g_{i}}\alpha^{-i}\overline{x}^{i})=\theta'(\overline{f})\theta(\overline{g}).
\end{align*}
Hence, the ring $T_{n}$ is well-defined by the definition of $\theta'$.
Now, it is turn to prove the isomorphism between $T_{n},R_{n}$. \begin{theorem}
$T_{n}\simeq R_{n}$. \end{theorem} \begin{proof} Let $\psi:R_{n}\longrightarrow T_{n}$
be as follows 
\begin{align*}
\psi(\sum_{i}(f_{i}+vl_{i}x^{i})+R(x^{n}-1))=(\sum_{i}(f_{i}x^{i})+F_{p}[x](x^{n}-1))+(\sum_{i}l_{i}x^{i}+F_{p}[x](x^{n}-1))v+<v^{2}>.
\end{align*}
First, we prove that $\psi$ is well-defined. Let $\sum_{i}(f_{i}+vl_{i}x^{i})+R(x^{n}-1)=\sum_{i}(g_{i}+vk^{i}x^{i})+R(x^{n}-1)$.
So $x^{n}-1|\sum_{i}((f_{i}-g_{i})+v(l_{i}-k_{i}))x^{i}$. Hence,
there exists $u,w\in\frac{F_{p}[x]}{<x^{n}-1>}$, such that 
\begin{align*}
(x^{n}-1)u+v(x^{n}-1)w=\sum_{i}(f_{i}-g_{i})x^{i}+v\sum_{i}(l_{i}-k_{i}))x^{i}
\end{align*}
Thus, $x^{n}-1|\sum_{i}(f_{i}-g_{i})x^{i}$ and $x^{n}-1|\sum_{i}(l_{i}-k_{i})x^{i}$.
So $\sum_{i}f_{i}x^{i}+F_{p}[x](x^{n}-1)=\sum_{i}g_{i}x^{i}+F_{p}[x](x^{n}-1)$
and $\sum_{i}l_{i}x^{i}+F_{p}[x](x^{n}-1)=\sum_{i}k_{i}x^{i}+F_{p}[x](x^{n}-1)$.
So $\psi(\sum_{i}(f_{i}+vl_{i})x^{i}+R(x^{n}-1))=\psi(\sum_{i}(g_{i}+vk_{i})x^{i}+R(x^{n}-1))$.
This proves that $\psi$ is well-defined.

Second, we prove that $\psi$ is a ring homomorphism. Let $u(x)=(f+vl)+R(x^{n}-1)\in R_{n}$
and $v(x)=(g+vk)+R(x^{n}-1)\in R_{n}$. One can see 
\begin{align*}
\overline{u}(x)\overline{v}(x)=\big((f+vl)+R(x^{n}-1)\big)\big((g+vk)+R(x^{n}-1)\big)=(fg+vlg+vf'k)+R(x^{n}-1).
\end{align*}
So 
\begin{align*}
 & \psi(\overline{u}(x))\psi(\overline{v}(x))\\
= & \big((f+F_{p}[x](x^{n}-1))+(l+F_{p}[x](x^{n}-1))v+<v^{2}>\big)\\?
&\times\big((g+F_{p}[x](x^{n}-1))+(k+F_{p}[x](x^{n}-1))v+<v^{2}>\big)\\
= & \big(fg+F_{p}[x](x^{n}-1)\big)+\big(f'k+F_{p}[x](x^{n}-1)\big)v+(lg+F_{p}[x](x^{n}-1))v+<v^{2}>.
\end{align*}
Hence, $\psi(\overline{u}\overline{v})=\psi(\overline{u})\psi(\overline{v})$.
Also, $\psi(\overline{u}+\overline{v})=\psi(\overline{u})+\psi(\overline{v})$
is easy to prove.

Third, we prove that $psi$ is an injective map. Let $\psi(\sum_{i}(f_{i}+vl_{i})x^{i}+R(x^{n}-1))=\overline{0}$.
So 
\begin{align*}
\big(\sum_{i}f_{i}x^{i}+F_{p}[x](x^{n}-1)\big)+v\big(\sum_{i}l_{i}x^{i}+F_{p}[x](x^{n}-1)\big)+<v^{2}>=\overline{0}.
\end{align*}
Thus, $u(x)(x^{n}-1)=\sum_{i}f_{i}x^{i}$, $w(x)(x^{n}-1)=\sum_{i}l_{i}x^{i}$
for some $u,w\in F_{p}[x]$. Hence, $(u(x)+vw(x))(x^{n}-1)=\sum_{i}(f_{i}+vl_{i})x^{i}$.
So $\sum_{i}(f_{i}+vl_{i})x^{i}+R(x^{n}-1)=0+R(x^{n}-1)$.

Finally, we prove that $\psi$ is surjective. Let $u(x)=h(x)+F_{p}[x](x^{n}-1)+(l(x)+F_{p}[x])v+<V^{2}>\in T_{n}$.
It is easy to see that 
\begin{align*}
\psi(h(x)+vl(x)+R(x^{n}-1))=u(x).
\end{align*}
\end{proof}

\begin{corollary} Every skew cyclic code over $F_{p}+vF_{p}$, $v^{2}=0$
and $\theta(v)=\alpha v$ is as the form of an ideal of $T_{n}$ where
$O(\theta)|n$. \end{corollary} So we can study the ideals of $T_{n}$
to get information about the skew cyclic codes over $F_{p}+vF_{p}$.

\subsection{Prime ideals of $T_{n}$}

Let $A\trianglelefteq R$, $x^{n}-1\in A$. So $\frac{A}{<x^{n}-1>}\trianglelefteq R_{n}$.
Thus, $\widehat{A}=\psi(\frac{A}{<x^{n}-1>})\trianglelefteq T_{n}$.
Let $\overline{f}=\big(g+F_{p}[x](x^{n}-1)\big)+\big(h+F_{p}[x](x^{n}-1)\big)v+<v^{2}>$.
Now, define 
\begin{align*}
\overline{A}_{[1]}=\{ & g+F_{p}[x](x^{n}-1)|\exists h+F_{p}[x](x^{n}-1)\in\frac{F_{p}[x]}{<x^{n}-1>},\\
 & \psi^{-1}\big(g+F_{p}[x](x^{n}-1)+v(h+F_{p}[x](x^{n}-1))+<v^{2}>\big)\in\frac{A}{<x^{n}-1>}\}\\
\overline{A}_{[2]}=\{ & h+F_{p}[x](x^{n}-1)|\exists g+F_{p}[x](x^{n}-1)\in\frac{F_{p}[x]}{<x^{n}-1>},\\
 & \psi^{-1}\big(g+F_{p}[x](x^{n}-1)+v(h+F_{p}[x](x^{n}-1))+<v^{2}>\big)\in\frac{A}{<x^{n}-1>}\}
\end{align*}

\begin{theorem}\label{raft} If $A\trianglelefteq R$ is a first
type ideal, and $x^{n}-1\in A$, then $\psi(\frac{A}{<x^{n}-1>})=\overline{A}_{[1]}+\overline{A}_{[2]}v+<v^{2}>$
(Consider $\overline{A}_{[1]},\overline{A}_{[2]}$ as subrings of
$T_{n}$). \end{theorem} \begin{proof} Let $A_{[1]}=<f>$ and $A_{[2]}=<g>$.
So $f+vg\in A$. Hence, $\psi(f+vg+R(x^{n}-1))\in\psi(\frac{A}{<x^{n}-1>})=\widehat{A}$.
So 
\begin{align*}
(f+F_{p}[x](x^{n}-1))+(g+F_{p}[x](x^{n}-1))v+<v^{2}>\in\widehat{A}.
\end{align*}
It is enough to show that $\overline{A}_{[1]}=<\overline{f}>,\overline{A}_{[2]}=<\overline{g}>$.
Let $k\in\widehat{A}$ and $k=(h+F_{p}[x](x^{n}-1))+(l+F_{p}[x](x^{n}-1))v+<v^{2}>$.
Hence, $\psi^{-1}(k)=h+vl+R(x^{n}-1)$. Thus, there exists $u,w\in F_{p}[x]$
such that 
\begin{align*}
h+vl+(x^{n}-1)(u+vw)\in A.
\end{align*}
So $h(x)+u(x)(x^{n}-1)\in A_{[1]}$ and $l(x)+w(x)(x^{n}-1)\in A_{[2]}$.
This means that $f(x)|h(x)+u(x)(x^{n}-1)$ and $g(x)|l(x)+w(x)(x^{n}-1)$.
Considering the fact that $f|x^{n}-1,g|x^{n}-1$, $f|h,g|l$. So $h=fh_{1}$
and $l=gl_{1}$.

Hence, $k=(fh_{1}+F_{p}[x](x^{n}-1))+v(gl_{1}+F_{p}[x](x^{n}-1))+<v^{2}>$.
Thus, $\overline{A}_{[1]}=<f+F_{p}[x](x^{n}-1)>$, $\overline{A}_{[2]}=<g+F_{p}[x](x^{n}-1)>$.
Considering the fact that $(f+F_{p}[x](x^{n}-1))+v(g+F_{p}[x](x^{n}-1))+<v^{2}>\in\widehat{A}$,
$\widehat{A}=\overline{A}_{[1]}+v\overline{A}_{[2]}+<v^{2}>$. \end{proof}
\begin{theorem}\label{bargasht} Let $A\trianglelefteq R$ , $x^{n}-1\in A$
and $\widehat{A}=\overline{A}_{[1]}+v\overline{A}_{[2]}+<v^{2}>$.
Then $A$ is a first type ideal of $R$. \end{theorem} \begin{proof}
Let $A_{[1]}=<f>$ and $A_{[2]}=<g>$. Suppose that $(h+F_{p}[x](x^{n}-1))+(k+F_{p}[x](x^{n}-1))v+<v^{2}>\in\widehat{A}$.
So $h+vk+R(x^{n}-1)\in\frac{A}{<x^{n}-1>}$. Hence, there exists $l_{1},l_{2}\in R$
such that $f|h+l_{1}(x^{n}-1)$ and $g|k+l_{2}(x^{n}-1)$. Thus $f|h$
and $g|k$. If $h=h_{1}f,k=k_{1}g$, $(h_{1}+F_{p}[x](x^{n}-1))(f+F_{p}[x](x^{n}-1))+(k_{1}+F_{p}[x](x^{n}-1))(g+F_{p}[x](x^{n}-1))v+<>v^{2}\in\widehat{A}$.
So $\overline{A}_{[1]}=<f+F_{p}[x](x^{n}-1)>$, $A_{[2]}=<g+F_{p}[x](x^{n}-1)>$.
This means that $(f+F_{p}[x](x^{n}-1))+(g+F_{p}[x](x^{n}-1))v+<v^{2}>\in\widehat{A}$
(Otherwise, $\widehat{A}\neq\overline{A}_{[1]}+\overline{A}_{[2]}v+<v^{2}>$).
Hence, $f+vg+R(x^{n}-1)=\psi^{-1}((f+F_{p}[x](x^{n}-1))+(g+F_{p}[x](x^{n}-1))v+<v^{2}>)\in\frac{A}{<x^{n}-1>}$.
So there exists $l_{1}\in R$ such that $f+vg+l_{1}(x^{n}-1)\in A$.
Considering the fact that $x^{n}-1\in A$, $f+vg\in A$. So $A=A_{[1]}+A_{[2]}v$.
Hence, $A$ is first type. \end{proof} We call $\widehat{A}$ is
a first type ideal of $T_{n}$, if $A$ is a first type ideal of $R$.
Hence, $\widehat{A}$ is first type, if and only if $\widehat{A}=\overline{A}_{[1]}+\overline{A}_{[2]}v+<v^{2}>$.
We show it by $\widehat{A}=\overline{A}_{[1]}+\overline{A}_{[2]}v$
for simplicity reasons. Also, we make two category of skew cyclic
codes of length $n$ over $S$. \begin{definition} Let $\complement$
be an skew cyclic code. Then $\complement$ is a first(second) skew
cyclic code, iff its correspondence ideal in $T_{n}$ is a first(second)
type ideal. \end{definition} \begin{theorem}\label{omid} Let $A\trianglelefteq R$,
$x^{n}-1\in A$. Then $\overline{A}_{[1]},\overline{A}_{[2]}$ are
ideals of $\frac{F_{p}[x]}{<x^{n}-1>}$. \end{theorem} \begin{proof}
Let $f+F_{p}[x](x^{n}-1)\in\overline{A}_{[1]}$ and $g+F_{p}[x](x^{n}-1)\in\frac{F_{p}[x]}{<x^{n}-1>}$,
So there exists $k+F_{p}[x](x^{n}-1)\in\frac{F_{p}[x]}{<x^{n}-1>}$
such that 
\begin{align*}
\psi^{-1}\big((f+F_{p}[x](x^{n}-1))+v(k+F_{p}[x](x^{n}-1))+<v^{2}>\big)\in\frac{A}{<x^{n}-1>}.
\end{align*}
Hence, $f+vk+R(x^{n}-1)\in\frac{A}{<x^{n}-1>}$. So $\big((g+R(x^{n}-1))(f+vk+R(x^{n}-1))\big)\in\frac{A}{<x^{n}-1>}$.
So $fg+vkg+R(x^{n}-1)\in\frac{A}{<x^{n}-1>}$. Thus 
\begin{align*}
\psi\big(fg+vkg+R(x^{n}-1)\big)\in\widehat{A}.
\end{align*}
Hence, $fg+F_{p}[x](x^{n}-1)\in\overline{A}_{[1]}$ which means that
$(f+F_{p}[x](x^{n}-1))(g+F_{p}[x](x^{n}-1))\in\overline{A}_{[1]}$.
Thus $\overline{A}_{[1]}$ is an ideal of $\frac{F_{p}[x]}{<x^{n}-1>}$.
In similar way, one can see that $\overline{A}_{[2]}$ is an ideal
of $\frac{F_{p}[x]}{<x^{n}-1>}$. \end{proof} \begin{theorem} Let
$A\trianglelefteq R$, $x^{n}-1\in A$. Then $\psi(v\overline{A}_{[1]})=\frac{\overline{A}_{[1]}[v,\theta']v}{<v^{2}>}\trianglelefteq T_{n}$.
Moreover, $\frac{\overline{A}_{[1]}[v,\theta']v}{<v^{2}>}\subseteq\widehat{A}$.
\end{theorem} \begin{proof} Let $f+F_{p}[x](x^{n}-1)\in\overline{A}_{[1]}$
and $\overline{u}=(g+F_{p}[x](x^{n}-1))+(k+F_{p}[x](x^{n}-1))v+<v^{2}>\in T_{n}$.
It is enough to show that $\overline{u}\big((f+F_{p}[x](x^{n}-1))v+<v^{2}>\big)\in\frac{\overline{A}_{[1]}[v,\theta']v}{<v^{2}>}$.
Since, $\overline{A}_{[1]}\trianglelefteq\frac{F_{p}[x]}{<x^{n}-1>}$
and $f+F_{p}[x](x^{n}-1)\in\overline{A}_{[1]}$, $fg\in\overline{A}_{[1]}$.
So $(g'f+F_{p}[x](x^{n}-1))v+<v^{2}>\in\frac{v\overline{A}_{[1]}[v,\theta']}{<v^{2}>}$.
Hence, 
\begin{align*}
 & \big((f+F_{p}[x](x^{n}-1))v+<v^{2}>\big)\big((g+F_{p}[x](x^{n}-1))+(k+F_{p}[x](x^{n}-1))v+<v^{2}>\big)\\
 & =\big((fg+F_{p}[x](x^{n}-1))v+<v^{2}>\big)\in\frac{\overline{A}_{[1]}[v,\theta']v}{<v^{2}>}.
\end{align*}
So $\frac{\overline{A}_{[1]}[v,\theta']v}{<v^{2}>}\trianglelefteq T_{n}$.
Also, let $f+F_{p}[x](x^{n}-1)\in\overline{A}_{[1]}$. So there exists
$h+F_{p}[x](x^{n}-1)\in\frac{F_{p}[x]}{<x^{n}-1>}$ such that $\psi^{-1}\big((f+F_{p}[x](x^{n}-1))+(h+F_{p}[x](x^{n}-1))v+<v^{2}>\big)\in\frac{A}{<x^{n}-1>}$.
Thus $f+vh+R(x^{n}-1)\in\frac{A}{<x^{n}-1>}$. Hence, $(v+R(x^{n}-1))(f+vh+R(x^{n}-1))\in\frac{A}{<x^{n}-1>}$.
So $vf+R(x^{n}-1)\in\frac{A}{<x^{n}-1>}$. This means that $\widehat{A}$.
So $(f+F_{p}[x](x^{n}-1))v+<v^{2}>\in\widehat{A}$. Thus $\frac{\overline{A}_{[1]}[v,\theta']v}{<v^{2}>}\subseteq\widehat{A}$.
\end{proof} \begin{theorem}\label{badbakhti} Let $A\trianglelefteq R$,
$x^{n}-1\in A$. Then $\overline{A}_{[1]}\subseteq\overline{A}_{[2]}$.
\end{theorem} \begin{proof} Let $f+F_{p}[x](x^{n}-1)\trianglelefteq\overline{A}_{[1]}$.
So there exists $h+F_{p}[x](x^{n}-1)\in\frac{F_{p}[x]}{<x^{n}-1>}$
such that $\psi^{-1}\big((f+F_{p}[x](x^{n}-1))+(h+F_{p}[x](x^{n}-1))v+<v^{2}>\big)\in\frac{A}{<x^{n}-1>}$.
Thus $f+vh+R(x^{n}-1)\in\frac{A}{<x^{n}-1>}$. So 
\begin{align*}
(v+R(x^{n}-1))(f+vh+R(x^{n}-1))=vf+R(x^{n}-1)\in\frac{A}{<x^{n}-1>}
\end{align*}
Thus $\psi(vf+R(x^{n}-1))\in\widehat{A}$. This means that $(f+F_{p}[x](x^{n}-1))v+<v^{2}>\in\widehat{A}$.
So $f+F_{p}[x](x^{n}-1)\in\overline{A}_{[2]}$. Hence, $\overline{A}_{[1]}\subseteq\overline{A}_{[2]}$.
\end{proof} \begin{theorem} If $P\trianglelefteq R$ is a prime
ideal and $x^{n}-1\in P$, then $\widehat{P}$ is a prime in $T_{n}$.
\end{theorem} \begin{proof} Let $\widehat{A}\widehat{B}\subseteq\widehat{P}$
and $\widehat{A},\widehat{B}$ are two arbitrary ideal of $T_{n}$.
Hence, $\psi^{-1}(\widehat{A}\widehat{B})\subseteq\psi^{-1}(\widehat{P})$.
Since $\psi$ is isomorphism, $\psi^{-1}(\widehat{A})\psi^{-1}(\widehat{B})\subseteq\psi^{-1}(\widehat{A}\widehat{B})$.
So $\psi^{-1}(\widehat{A})\psi(\widehat{B})\subseteq\frac{P}{<x^{n}-1>}$.
Thus, $\frac{A}{<x^{n}-1>}\frac{B}{x^{n}-1}\subseteq\frac{P}{x^{n}-1}$.
So $AB\subseteq P$. This implies $A\subseteq P$ or $B\subseteq P$
(Since $x^{n}-1\in A,B$). So $\psi(\frac{A}{<x^{n}-1>})\subseteq\widehat{P}$
or $\psi(\frac{B}{<x^{n}-1>})\subseteq\widehat{P}$. Hence, $\widehat{A}\subseteq\widehat{P}$
or $\widehat{B}\subseteq\widehat{P}$. 
\end{proof} \begin{theorem} Let $\widehat{P}$ is a prime ideal
of $T_{n}$, then $P$ is a prime ideal of $R$. \end{theorem} \begin{proof}
Let $AB\subseteq P,A,B\trianglelefteq R$. Suppose that $A^{*}=<A,x^{n}-1>$
and $B^{*}=<B,x^{n}-1>$. Since $x^{n}-1\in P$, $A^{*}B^{*}\subseteq P$.
So $\frac{A^{*}}{<x^{n}-1>}\frac{B^{*}}{<x^{n}-1>}\subseteq\frac{P}{<x^{n}-1>}$.
So $\widehat{A^{*}}\widehat{B^{*}}\subseteq\widehat{P}$. This implies
that $\widehat{A^{*}}\subseteq\widehat{P}$ or $\widehat{B^{*}}\subseteq\widehat{P}$.
So $\frac{A^{*}}{<x^{n}-1>}\subseteq\frac{P}{<x^{n}-1>}$ or $\frac{B^{*}}{<x^{n}-1>}\subseteq\frac{P}{<x^{n}-1>}$.
Hence, $A^{*}\subseteq P$ or $B^{*}\subseteq P$. This means that
$A\subseteq P$ or $B\subseteq P$. So $P$ is a prime ideal of $R$.
\end{proof} \begin{corollary} Let $P\trianglelefteq R$, $x^{n}-1\in P$.
Then $P$ is a prime ideal of $R$, iff $\widehat{P}$ is a prime
ideal of $T_{n}$. \end{corollary} \begin{corollary} Let $\widehat{P}\trianglelefteq T_{n}$
be a prime ideal. Then $P=\psi\big(\frac{F_{p}[x]f+vF_{p}[x]}{<x^{n}-1>}\big)$
where $f\in F_{p}[x]$ is an irreducible polynomial such that $f|x^{n}-1$.
\end{corollary}

\subsection{The primary ideals of $T_{n}$}

First, we start with some lemma to find an equivalence theorem between
primary ideals of $T_{n}$ and some of primary ideals in $R$. \begin{lemma}\label{lem khubu}
Let $A\trianglelefteq R,x^{n}-1\in A$. Then $(\psi^{-1}(\widehat{A}))^{m}\subseteq\psi^{-1}(\widehat{A}^{m})$.
\end{lemma} \begin{proof} Since $\psi$ is an isomorphism, $\psi^{-1}(B)\psi^{-1}(C)\subseteq\psi^{-1}(BC)$
for each ideals of $T_{n}$ like $B,C$. So $\psi^{-1}(\widehat{A})^{m}\subseteq\psi^{-1}(\widehat{A}^{m})$.
\end{proof} Also, one can prove that easily that $\frac{A^{m}}{<x^{n}-1>}=(\frac{A}{<x^{n}-1>})^{m}$.
\begin{theorem} Let $Q\trianglelefteq R$, $x^{n}-1\in Q$. If $\widehat{Q}$
is a primary ideal of $T_{n}$, then $Q$ is a primary ideal of $R$.
\end{theorem} \begin{proof} Let $AB\subseteq Q$. Suppose that $A^{*}=<A,x^{n}-1>,B^{*}=<B,x^{n}-1>$.
So $A^{*}B^{*}\subseteq Q$ and this results in $\frac{A^{*}}{<x^{n}-1>}\frac{B^{*}}{<x^{n}-1>}\subseteq\frac{Q}{<x^{n}-1>}$.
Hence, $\psi(\frac{A^{*}}{<x^{n}-1>})\psi(\frac{B^{*}}{<x^{n}-1>})\subseteq\psi(\frac{Q}{<x^{n}-1>})$.
Thus $\widehat{A^{*}}\widehat{B^{*}}\subseteq\widehat{Q}$. This means
that $\widehat{A^{*}}\subseteq\widehat{Q}$ or $(\widehat{B^{*}})^{m}\subseteq\widehat{Q}$
for some $m\in\mathbb{N}$. Hence, $\psi^{-1}(\widehat{A^{*}})\subseteq\psi^{-1}(\widehat{Q})$
or $\psi^{-1}(\widehat{B^{*}}^{m})\subseteq\psi^{-1}(\widehat{Q})$.
So $\frac{A^{*}}{<x^{n}-1>}\subseteq\frac{Q}{<x^{n}-1>}$ or $(\frac{B^{*}}{<x^{n}-1>})^{m}\subseteq\frac{Q}{<x^{n}-1>}$
for some $m\in\mathbb{N}$ by lemma \ref{lem khubu}. So $A^{*}\subseteq Q$
or $(B^{*})^{m}\subseteq Q$ for some $m\in\mathbb{N}$. So $Q$ is
primary. 
\end{proof} 
\begin{theorem} Let $Q\trianglelefteq R,x^{n}-1\in Q$ be a primary
ideal of $R$. Then $\widehat{Q}$ is a primary ideal of $T_{n}$.
\end{theorem} \begin{proof} Let $\widehat{A}\widehat{B}\subseteq\widehat{Q}$.
So $\psi^{-1}(\widehat{A})\psi^{-1}(\widehat{B})\subseteq\psi^{-1}(\widehat{A}\widehat{B})\subseteq\psi(\widehat{Q})$.
So $\frac{AB,x^{n}-1}{<x^{n}-1>}\subseteq\frac{Q}{<x^{n}-1>}$. Hence,
$<AB,x^{n}-1>\subseteq Q$. Thus $AB\subseteq Q$ which results in
$A\subseteq Q$ or $B^{m}\subseteq Q$ fr some $m\in\mathbb{N}$.
So $\psi(\frac{A}{<x^{n}-1>})\subseteq\psi(\frac{Q}{<x^{n}-1>})$
or $\psi(\frac{B^{m},x^{n}-1}{<x^{n}-1>})=\big(\psi(\frac{B}{<x^{n}-1>})\big)^{m}\subseteq\psi(\frac{Q}{<x^{n}-1>})$
for some $m\in\mathbb{N}$. Hence, $\widehat{A}\subseteq\widehat{Q}$
or $\widehat{B}^{m}\subseteq\widehat{Q}$ for some $m\in\mathbb{N}$.
Thus $\widehat{Q}$ is a primary ideal. \end{proof} \begin{corollary}\label{primary}
Let $Q\trianglelefteq R,x^{n}-1\in Q$. Then $Q$ is a primary ideal
of $Q$, iff $\widehat{Q}$ is a primary ideal of $T_{n}$. In particular,
every primary ideal $\widehat{Q}$ in $T_{n}$ is the first type ideal
and exactly in one of the following forms.

i) $\psi\big(\frac{F_{p}[x]f^{a}+vF_{p}[x]}{<x^{n}-1>}\big)$ where
$f\in F_{p}[x]$ is an irreducible polynomial such that $f^{a}|x^{n}-1$
and $a\geq0$.

ii) $\psi\big(\frac{F_{p}[x]f^{a}+vF_{p}[x]f^{b}}{<x^{n}-1>}\big)$
where $f\in F_{p}[x]$ is an irreducible polynomial such that $f^{b}|x^{n}-1$
and $a>b\geq0$.

iii) $\widehat{0}$. \end{corollary}

\begin{theorem}\label{tajzie} Let $A\trianglelefteq R,x^{n}-1\in A$.
If $\widehat{A}$ has a primary decomposition such that all of its
primary coefficients in decomposition are first type, then $\widehat{A}$
is a first type ideal. \end{theorem} \begin{proof} Let $\widehat{A}$
be a second type ideal. Also $\widehat{A}=\widehat{Q}_{1}\cap\widehat{Q}_{2}\cap\cdots\cap\widehat{Q}_{t}$
for some primary ideals $\widehat{Q}_{i}$. Then $\psi^{-1}(\widehat{A})=\psi^{-1}(\widehat{Q}_{1})\cap\cdots\cap\psi^{-1}\widehat{Q}_{t}$.
So $\frac{A}{<x^{n}-1>}=\frac{Q_{1}}{<x^{n}-1>}\cap\cdots\cap\frac{Q_{t}}{<x^{n}-1>}$.
As$x^{n}-1\in Q_{i}$ for some $1\leq i\leq t$, $x^{n}-1\in\bigcap_{i=1}^{t}Q_{i}$.
So $\frac{A}{<x^{n}-1>}=\frac{\bigcap_{i=1}^{t}Q_{i}}{<x^{n}-1>}$.
Hence, $A=Q_{1}\cap\cdots\cap Q_{t}$. Thus $A$ should be a first
type ideal by lemma \ref{decomposition}. So $\widehat{A}$ is a first
type ideal by lemma \ref{raft}. \end{proof}

So if $\widehat{A}$ is a second type ideal and has a primary decomposition,
then there exists at least one second type primary coefficient in
its decomposition. But finding a second type ideal is not easy and
from the computation view, it seems demanding. 

\begin{corollary} All of skew cyclic codes like $\complement$ of
length $n$ over $\mathbb{F}_{p}+\mathbb{F}_{p}$ are in exactly one
of the following forms (One can transform these forms to the ideals
of $R_{n}$).

i) $\complement=\bigcap_{i}\psi\big(\frac{F_{p}[x]f_{i}^{a_{i}}+vF_{p}[x]}{<x^{n}-1>}\big)$
where $f_{i}\in F_{p}[x]$ are irreducible polynomials such that $f_{i}^{a_{i}}|x^{n}-1$
and $a_{i}\geq0$.

ii) $\complement=\bigcap_{i}\psi\big(\frac{F_{p}[x]f_{i}^{a_{i}}+vF_{p}[x]f_{i}^{b_{i}}}{<x^{n}-1>}\big)$
where $f_{i}\in F_{p}[x]$ are irreducible polynomials such that $f_{i}^{b_{i}}|x^{n}-1$
and $a_{i}>b_{i}\geq0$.

iii) $\complement=\widehat{0}$.


iv) For each $f_{2}|x^{n}-1$ and $g\in R$ such that $x^{n}-1\in Rg+v\mathbb{F}_{p}[x]f_{2}$,
then \\
 $\complement=\psi\bigg(R_{n}(g+R(x^{n}-1))+v\frac{\mathbb{F}_{p}[x]}{<x^{n}-1>}(f_{2}+\frac{\mathbb{F}_{p}[x]}{<x^{n}-1>})\bigg)$
(If the leading coefficient of the minimum degree polynomial $vf_{2}$
is non unit and $g$ is the polynomial in $\complement$ with the
unit leading coefficient such that has the least degree).

v) For each $f$ such that $f|x^{n}-1$ for some $g\in R$, $\complement=\psi(R_{n}(f+R(x^{n}-1)))$
(If the leading coefficient of the minimum degree polynomial $f$
is unit). \end{corollary} \begin{proof} Note that if $g\in R$ and
$x^{n}-1=gh$ for some $h$, then $g\in\mathbb{F}_{p}[x]$ and $g|x^{n}-1$.
The rest is the result of \ref{primary}, \ref{tajzie} and \ref{secondd}.
\end{proof} Assume a first type ideal $\widehat{A}$. So $\widehat{A}=\overline{A}_{[1]}+v\overline{A}_{[2]}$.
Since $\frac{F_{p}[x]}{<x^{n}-1>}$ is a notherian commutative ring,
the unique primary decomposition exists for $\overline{A}_{[1]},\overline{A}_{[2]}$.
So 
\begin{align}
\widehat{A}=\big(\bigcap_{i}Z_{i}\big)+v\big(\bigcap_{i}Y_{i}\big)
\end{align}
where $Z_{i},Y_{i}$ are primary ideals of $\frac{F_{p}[x]}{<x^{n}-1>}$.
So there is a characterization for first type ideals of $T_{n}$ which
means a characterization for skew cyclic codes of length $n$ over
$S$. Let $\complement=\complement_{1}+v\complement_{2}$ be a first
skew cyclic code. We proved that each $\complement_{1},\complement_{2}$
are in fact two cyclic codes over $F_{p}$ (see \ref{omid}). So there
exists two matrices $G_{1},G_{2}$ which correspond to $\complement$.
Hence, there exists two parity check matrices like $H_{1},H_{2}$
for $\complement$.

\section{General properties of skew cyclic codes of length $n$ over $S$}

In this section, we prove some properties of skew cyclic codes of
length $n$ like $\complement$ over $S$. Note that $O(\theta)|n$
is not necessarily holds in this section. First, we prove that $(\mathbb{F}_{p}+v\mathbb{F}_{p})[x,\theta]$
is a free $\mathbb{F}_{p}$-module. \begin{theorem} $\frac{(\mathbb{F}_{p}+v\mathbb{F}_{p})[x,\theta]}{<x^{n}-1>}$
is a free $\mathbb{F}_{p}$-module with following basis $A$. 
\begin{align}
A=\{1,x,\cdots,x^{n-1}\}\cup\{v,vx,\cdots,vx^{n-1}\}
\end{align}
\end{theorem} \begin{proof} Let $f\in R$, then $f=\sum_{i=0}^{k}f_{i}x^{i}$
for $k\leq n-1$ and $f_{i}\in\mathbb{F}_{p}+v\mathbb{F}_{p}$. If
$f_{i}=f_{1,i}+vf_{2,i}$, then 
\begin{align}
f(x)=\sum_{i=0}^{k}f_{1,i}x^{i}+\sum_{i=0}^{k}vf_{2,i}x^{i}=\sum_{i=0}^{k}f_{1,i}x^{i}+\sum_{i=0}^{k}f'_{2,i}vx^{i}
\end{align}
So $f\in<1,x,\cdots,x^{n-1},v,vx,\cdots,vx^{n-1}>$ which means that
$R=<1,x,\cdots,x^{n-1},v,vx,\cdots,vx^{n-1}>$.

Now suppose that 
\begin{align}
\sum_{i=0}^{n-1}a_{i}x^{i}+\sum_{i=0}^{n-1}b_{i}vx^{i}=0
\end{align}
By multipling the above equation to $v$, it is easily concluded that
$\sum_{i=0}^{n-1}a_{i}x^{i}=0$. So it results in $\sum_{i=0}^{n-1}b_{i}vx^{i}=0$.
Now we know that both sets $\{x^{i}\}$ and $\{vx^{i}\}$ are independent
which means that $a_{i}=b_{i}=0$. So $\{1,x,\cdots,x^{n-1},v,vx,\cdots,vx^{n-1}\}$
is an independent set. \end{proof} Also we know that each first type
skew cyclic code is in the form $\complement=\complement_{1}+v\complement_{2}=<fg>+v<f>$
by \ref{badbakhti}. So both part $\complement_{1},\complement_{2}$
are $\mathbb{F}_{p}$-submodule of $\mathbb{F}_{p}[x]$. In the next
theorem, we will introduce their basis. \begin{theorem} All first
type skew cyclic codes are $\mathbb{F}_{p}$-submodule of $R_{n}=\frac{\mathbb{F}_{p}[x]+v\mathbb{F}_{p}[x]}{<x^{n}-1>}$
with the following basis $B$. 
\begin{align}
B=\{(x^{i})|0\leq i\leq n-deg(fg)\}\cup\{vx^{i}|0\leq i\leq n-deg(f)\}
\end{align}
\end{theorem} \begin{proof} First we show that $\complement$ is
$\mathbb{F}_{p}$-submodule. Let $a\in\mathbb{F}_{p}$ and $h,k\in\complement$.
Linearity of $\complement$ follows that $h+ak\in\complement$ which
means $\complement$ is $\mathbb{F}_{p}$-submodule.

Now let $h\in\complement$. So $h=h_{1}+vh_{2}$ where $h_{1}\in\complement_{1}$
and $h_{2}\in\complement_{2}$. Thus $fg|h_{1}$ and $f|h_{2}$. So
$h_{1}=kfg$ and $vh_{2}=vlf$ for some $l,k$. So $deg(k)\leq deg(h_{1})-deg(fg)$
and $deg(l)\leq deg(h_{2})-deg(f)$ which means that $k\in<1,x,\cdots,x^{n-deg(fg)}>$
and $l\in<1,x,\cdots,x^{n-deg(f)}>$. Also $x^{n-deg(fg)}\in\complement_{1}$
and $x^{n-deg(f)}\in\complement_{2}$. So $\complement_{1}=<1,x,\cdots,x^{n-deg(fg)}>$
and $\complement_{2}=<v,vx,\cdots,vx^{n-deg(f)}>$. So $\complement=<1,x,\cdots,x^{n-deg(fg)},v,vx,\cdots,vx^{n-deg(f)}>$.
It is easy to see that $\{x^{i}\}\cup\{vx^{j}\}$ is independent set.
\end{proof} Let $\complement$ be a first type code. A lot of properties
of cyclic codes still remains with proper changes. One of them is
minimum distace as follows.

\subsection{skew cyclic codes over $S$ in the case $O(\theta)\nmid n$}

If $O(\theta)\nmid n$, then $<x^{n}-1>$ is not a two sided ideal.
So the set $R_{n}$ is not a ring. But it is a $\mathbb{F}_{p}$-module
according to later discussion. Also we proposed the folloing theorem
for skew cyclic codes. Its proof is the same as theorem 3.5 in \cite{gao2013}.
\begin{theorem} Let $O(\theta)\nmid n$. Then $\complement$ is a
skew cyclic code of length $n$ over $\mathbb{F}_{p}+v\mathbb{F}_{p}$
if and only if $\complement$ is a left sub-module of $R_{n}$. \end{theorem}
We know that an skew cyclic code over $\mathbb{F}_{p}+v\mathbb{F}_{p}$
like $\complement$ is a quasi cyclic code with index $O(\theta)$.
The proof of the following proposition is inspired from the proof
of \cite[Theorem 3.7.]{gao2013}. \begin{proposition} Let $O(\theta)\nmid n$
and $\complement$ be a skew cyclic code of length $n$ and $\gcd(n,O(\theta))=d$.
Then $\complement$ is equivalent to a quasi cyclic code of length
$n$ over $\mathbb{F}_{p}+v\mathbb{F}_{p}$ with index $d$. \end{proposition}
\begin{proof} Let $O(\theta)=e$. We know that $\gcd(n,e)=d$ which
means that $ae-ln=d$ for some integer $a$ and $l>0$. Let $c(x)=\sum_{i=0}^{n-1}c_{i}x^{i}$
be a codeword in $\complement$. Then we know that $x^{ae}c(x)\in\complement$.
But 
\begin{align}
x^{ae}c(x)=\sum_{i=0}^{n-1}\theta^{ae}(c_{i})x^{d+ln+i}=\sum_{i=0}^{n-1}c_{i}x^{d+i}\in\complement
\end{align}
So $\complement$ is equivalent to a quasi cyclic code with index
$d$. \end{proof} We know that each skew cyclic code over $\complement$
can be considered as $\complement=\complement_{1}\bigoplus v\complement_{2}$
which both of them are ideals in $\mathbb{F}_{p}[x]$. So we can count
them and get the following theorem. \begin{theorem} Let $O(\theta)\nmid n$,
then the number of distinct skew cyclic codes over $\mathbb{F}_{p}+v\mathbb{F}_{p}$
with length $n$ is equal to the following. 
\begin{align}
\prod_{i=1}^{s}(r_{i}+1)^{2}
\end{align}
where $x^{n}-1=\prod_{i=1}^{s}p_{i}^{r_{i}}(x)$. \end{theorem}

\section{Some example}

In this section, we provide some examples of skew cyclic codes as
follows. \begin{example} We want to find all first type skew cyclic
codes with length $4$ over $F_{3}+vF_{3}$ with $\theta(v)=\alpha v$.
For this, one can see that the composition of $x^{4}-1$ in $F_{3}$
is as follows. 
\begin{align}
x^{4}-1=(x+2)(x+1)(x^{2}+1)
\end{align}
Also all of first type codes $\complement_{i}$ are in the following
form. 
\begin{align}
\complement_{i}=<fg>+v<f>
\end{align}
Where $fg|x^{4}-1$. Thanks to some basic concepts of counting, one
can see that we can have $24$ different first type skew cyclic codes.
For example, Let $f=x+2$ and $g=x^{2}+1$, then we can have code
$\complement$ as follows. 
\begin{align}
\complement=<(x+2)(x^{2}+1)>+v<x+2>
\end{align}
The generator matrices of this code like $G_{1},G_{2}$ are 
\begin{align}
G_{1}= & \begin{bmatrix}2 & 1 & 2 & 1\end{bmatrix}\nonumber \\
G_{2}= & \begin{bmatrix}2 & 1 & 0 & 0\\
0 & 2 & 1 & 0\\
0 & 0 & 2 & 1
\end{bmatrix}^{T}
\end{align}
Also its parity matrices $H_{1},H_{2}$ are as follows. 
\begin{align}
H_{1}= & \begin{bmatrix}1 & 1 & 0 & 0\\
0 & 1 & 1 & 0\\
0 & 0 & 1 & 1
\end{bmatrix}\nonumber \\
H_{2}= & \begin{bmatrix}1 & 2 & 2 & 1\end{bmatrix}^{T}
\end{align}
This code has minimim hamming distance $4$. Because there is not
a zero column. Also None of two or three column ar not dependent.
All of first type skew cyclic codes of length $4$ over $F_{3}+vF_{3}$
are designed in Appendix. \end{example} \begin{example} Now we will
find one of second type codes over $F_{3}+vF_{3}$. One can see that
$<h>=<x-1>$ is a cyclic codes over $F_{3}+vF_{3}$, Because 
\begin{align}
x^{4}-1=((x+1)(x^{2}+1)+v(x+1)(x^{2}+1))h
\end{align}
Since $v(x-1)=(\alpha x-1)v$ and $\alpha x-1$ is not unit, then
$<x-1>$ is a second type ideal by \ref{second}. So $<h>$ is a second
type skew cyclic code. \end{example} \begin{example} We find all
of first type skew cyclic codes with length $6$ over $F_{5}+vF_{5}$.
First, it is easy to see that 
\begin{align}
x^{6}-1=(x+1)(x+4)(x^{2}+x+1)(x^{2}+4x+1)
\end{align}
where all of the right side polynomials are irreducible. Also we know
that $\complement_{i}=<fg>+v<f>$. So there are $65$ first type skew
cyclic codes over $F_{5}+vF_{5}$. For example $\complement=<(x+1)(x+4)(x^{2}+4x+1)>+v<(x+1)(x^{2}+4x+1)>$
is one of these codes.
\end{example}

\section{Conclusion}

We studied construction and charcteristics of cyclic codes over \textbf{$\mathbb{F}_{p}+v\mathbb{F}_{p}$.}
We proved several theorems and studied distance properties of these
codes. We also provided some examples of such codes. This work can be extended to the ring
$(\mathbb{F}_{p}+v\mathbb{F}_{p}+\cdots+v^{n-1}\mathbb{F}_{p})[x;\theta]$
where $v^{n}=0$. 


\section{Appendix}

\begin{theorem} There exists exactly $p$ ring homomorphism from
$R$ to $R$. \end{theorem} \begin{proof} Let $\theta:R\longrightarrow R$
be a ring hemimorphism. So for each $a,b,c,d\in\mathbb{F}_{p}[x]$
\begin{align}
\theta(a+vb)\theta(c+vd)=\theta((a+vb)(c+vd))=\theta(ac+v(bc+ad))
\end{align}
It is easy to prove that $\theta(1)=1$. Also let $\theta(v)=x+vy$.
So 
\begin{align}
\theta(a+vb) & =a+(x+vy)b\nonumber \\
\theta(c+vd) & =c+(x+vy)d\nonumber \\
\theta(ac+v(bc+ad)) & =ac+(x+vy)(bc+ad)
\end{align}
Thus 
\begin{align}
\left(a+(x+vy)b\right)\left(c+(x+vy)d\right)=\left(ac+(x+vy)(bc+ad)\right)
\end{align}
So 
\begin{align}
ac+bcx+adx+bdx^{2}+v(bcy+bdxy+ady+bdxy)=ac+bcx+adx+v(ybc+yad)
\end{align}
Solving above equation for each $a,b,c,d$, ends in $x=0$ and arbitrary
$y$. So $\theta(a+vb)=a+vby$ for arbitrary $y\in\mathbb{F}_{p}$.
So there exists exactly $p$ hemimorphisms from $R$ to $R$. Moreover,
$\theta$ is automorphism, if and only if $y\neq0$. \end{proof}
\begin{corollary} There exists exactly trivial automorphism in $F_{2}+vF_{2}$
where $v^{2}=0$. \end{corollary} \begin{proof} Because of last
Theorem, $\theta(v)=yv$ and if $y=0$, then $\theta=0$ which can
not be surjective. Else if $\theta$ must be identity. So skew polynomial
ring $(F_{2}+vF_{2})[x,\theta]$ where $v^{2}=0$ is commutative polynomial
ring $(F_{2}+vF_{2})[x]$. Hence we just can have cyclic codes over
$(F_{2}+vF_{2})$ which is characterized by {[}Siap{]}. \end{proof}
In this part, we will introduce all of the first type skew cyclic
codes with length $4$ over $F_{3}+vF_{3}$ as the following proposition.
Before the following proposition, it should be noted that by theorem
10, all of these codes are in the form $<fg>+v<f>$ for some $f,g\in\mathbb{F}_{p}[x]$.
\begin{proposition} All of first type skew cyclic codes
with length $4$ over $F_{3}+vF_{3}$ are as follows. 
\begin{align}
 & \complement_{1}=<1>+v<1>\\
 & \complement_{2}=<x+2>+v<1>\\
 & \complement_{3}=<x+1>+v<1>\\
 & \complement_{4}=<x^{2}+1>+v<1>\\
 & \complement_{5}=<x+2>+v<x+2>\\
 & \complement_{6}=<x+1>+v<x+1>\\
 & \complement_{7}=<x^{2}+1>+v<x^{2}+1>\\
 & \complement_{8}=<(x+1)(x^{2}+1)>+v<1>\\
 & \complement_{9}=<(x+1)(x+2)>+v<1>\\
 & \complement_{10}=<(x+2)(x^{2}+1)>+v<1>\\
 & \complement_{11}=<(x+1)(x^{2}+1)>+v<x+1>\\
 & \complement_{12}=<(x+2)(x+1)>+v<x+1>\\
 & \complement_{13}=<(x+2)(x^{2}+1)>+v<x+2>\\
 & \complement_{14}=<(x+1)(x+2)>+v<x+2>\\
 & \complement_{15}=<(x+2)(x^{2}+1)>+v<x^{2}+1>\\
 & \complement_{16}=<(x+1)(x^{2}+1)>+v<x^{2}+1>\\
 & \complement_{17}=<x^{4}-1>+v<1>\\
 & \complement_{18}=<x^{4}-1>+v<x+1>\\
 & \complement_{19}=<x^{4}-1>+v<x^{2}+1>\\
 & \complement_{20}=<x^{4}-1>+v<(x+1)\\
 & \complement_{21}=<x^{4}-1>+v<(x+1)(x^{2}+1)>\\
 & \complement_{22}=<x^{4}-1>+v<(x^{2}+1)(x+2)>\\
 & \complement_{23}=<x^{4}-1>+v<(x+1)(x+2)>\\
 & \complement_{24}=<x^{4}-1>+v<x^{4}-1>
\end{align}
\end{proposition} 
\end{document}